\documentclass[11pt]{article}
\usepackage{amsmath, amssymb, amsthm}
\usepackage{verbatim}

\date{}
\author{
Michael Krivelevich \thanks{School of Mathematical Sciences, Raymond
and Beverly Sackler Faculty of Exact Sciences, Tel Aviv University,
Tel Aviv 69978, Israel. Email address: krivelev@post.tau.ac.il.
Research supported in part by USA-Israel BSF Grant 2010115 and by
grant 912/12 from the Israel Science Foundation. } \and Choongbum Lee
\thanks{Department of Mathematics, MIT, Cambridge, MA, 02142.
Email: cb\_lee@math.mit.edu. Research supported in part by a
Samsung Scholarship.} \and Benny Sudakov
\thanks{Department of Mathematics, UCLA, Los Angeles, CA 90095.
Email: bsudakov@math.ucla.edu. Research supported in part by 
NSF grant DMS-1101185, by AFOSR MURI grant FA9550-10-1-0569 and by a 
USA-Israel BSF grant. } }

\numberwithin{equation}{section}
\numberwithin{figure}{section}

\theoremstyle{plain}
\newtheorem{thm}{Theorem}[section]
  \theoremstyle{plain}
  \newtheorem{lem}[thm]{Lemma}
  \theoremstyle{remark}
  \newtheorem{claim}[thm]{Claim}

\newcommand{\BBP}{\mathbb{P}}
\newcommand{\BBE}{\mathbb{E}}

\begin{document}

\title{\vspace{-1cm}Long paths and cycles in random subgraphs of graphs with large minimum degree}

\maketitle

\begin{abstract}
For a given finite graph $G$ of minimum degree at least $k$, let
$G_{p}$ be a random subgraph of $G$ obtained by taking each edge
independently with probability $p$. We prove that (i) if $p \ge
\omega/k$ for a function $\omega=\omega(k)$ that tends to infinity
as $k$ does, then $G_p$ asymptotically almost surely contains a
cycle (and thus a path) of length at least $(1-o(1))k$, and (ii) if
$p \ge (1+o(1))\ln k/k$, then $G_p$ asymptotically almost surely
contains a path of length at least $k$. Our theorems extend
classical results on paths and cycles in the binomial random graph,
obtained by taking $G$ to be the complete graph on $k+1$ vertices.
\end{abstract}

\section{Introduction}

Paths and cycles are two of the most simple yet important structures
in graph theory, and being such, the problem of finding conditions
that imply the existence of paths and cycles of various lengths has
attracted a lot of attention in the field for the past 60 years. For
example, Dirac \cite{Dirac} proved that for $k \ge 2$, every graph
of minimum degree at least $k$ contains a path of length $k$ and a
cycle of length at least $k+1$, and that every graph on $n$ vertices
of minimum degree at least $\frac{n}{2}$ is Hamiltonian, i.e.,
contains a cycle of length $n$. By considering a complete graph on
$n=k+1$ vertices, and two edge-disjoint complete graphs of the same
size sharing a single vertex, one can see that these results are
tight. One reason that finding such conditions is of considerable
interest is because it is often the case that conditions implying
the existence of paths and cycles can be generalized to other
substructures such as trees, and general subgraphs. For example,
P\'osa \cite{Posa} proved that expansion implies the existence of
long paths, and later Friedman and Pippenger \cite{FrPi} generalized
P\'osa's approach to trees.

Given a graph $G$ and a real $p\in[0,1]$, let $G_{p}$ be the probability
space of subgraphs of $G$ obtained by taking each edge of $G$ independently
with probability $p$. We sometimes use the notation $(G)_p$ to avoid
ambiguity. For a given graph property $\mathcal{P}$,
and sequences of graphs $\{G_{i}\}_{i=1}^{\infty}$ and of probabilities
$\{p_{i}\}_{i=1}^{\infty}$, we say that $(G_{i})_{p_{i}}\in\mathcal{P}$
\emph{asymptotically almost surely}, or \emph{a.a.s.} for brevity,
if the probability that $(G_{i})_{p_{i}}\in\mathcal{P}$ tends to
$1$ as $i$ goes to infinity. In this paper, when $G$ and $p$ are
parameterized by some parameter, we abuse notation and consider $G$
and $p$ as sequences obtained by taking the parameter to tend to
infinity, and will say that $G_{p}$ has $\mathcal{P}$ asymptotically
almost surely if the sequence does.

The most studied case of the above model of random graphs is when
$G$ is a complete graph $K_{n}$ on $n$ vertices, where $(K_{n})_{p}$
is known as the \emph{binomial random graph} $G(n,p)$. This model
has been first introduced in 1959 \cite{Gilbert} and has been
intensively studied thereafter. In a seminal paper, Ajtai, Koml\'os,
and Szemer\'edi \cite{AjKoSz}, confirming a conjecture of Erd\H{o}s,
proved in particular that for $p=\frac{c}{n}$ and $c>1$, $G(n,p)$
a.a.s. contains a path of length at least $(1-f(c))n$, where $f(c)$
is a function tending to zero as $c$ goes to infinity; this was also
proved independently by Fernandez de la Vega \cite{Vega}. Analogous
problem for cycles was studied by Bollob\'as, Fenner, and Frieze
\cite{BoFeFr}. Frieze \cite{Frieze} later determined the asymptotics
of the number of vertices not covered by a longest path and cycle in
$G(n,p)$. Also, for Hamiltonian paths and cycles, i.e., paths and
cycles which pass through every vertex of the graph, improving on
results of P\'osa \cite{Posa} and Korshunov \cite{Korshunov},
Bollob\'as \cite{Bollobas} and Koml\'os and Szemer\'edi \cite{KoSz}
independently proved that for every fixed positive $\varepsilon$ and
$p\ge\frac{(1+\varepsilon)\log n}{n}$, the random graph $G(n,p)$ is
a.a.s. Hamiltonian. See the book of Bollob\'as \cite{Bollobas2} for
a comprehensive overview of results on paths and cycles in random
graphs.

In this paper we study generalizations of the above mentioned
results. Our goal is to extend classical results on random graphs to
a more general class of graphs. More precisely, we would like to
replace the host graph, taken to be the complete graph in the
classical setting, by a graph of large minimum degree, and to find
a.a.s. long paths and cycles in random subgraphs of large minimum
degree graphs. Throughout the paper, we will only consider finite
graphs.

Our first two theorems study paths.  
\begin{thm}
\label{thm:linearpath}Let $G$ be a finite graph with minimum degree at least $k$,
and let $p=\frac{c}{k}$ for some positive $c$ satisfying $c = o(k)$ 
($c$ is not necessarily fixed). Then a.a.s. $G_{p}$ contains a path of length $(1-2c^{-1/2})k$.
\end{thm}
We can also a.a.s.~find a path of length exactly $k$, given that $p$
is sufficiently large.
\begin{thm}
\label{thm:hamiltonpath}
Let $\varepsilon$ be a fixed positive real. 
For a finite graph $G$ of minimum degree at least $k$ and a real
$p\ge\frac{(1+\varepsilon)\log k}{k}$, $G_{p}$ a.a.s. contains a path of length
$k$.
\end{thm}

Note that we parameterize our
graph in terms of its minimum degree. Hence it should be
understood that there are underlying sequences of graphs $\{G_i\}_{i=1}^{\infty}$
and probabilities $\{p_i\}_{i=1}^{\infty}$ where the minimum degrees of
graphs tend to infinity, and the statements above hold with probability
tending to $1$ as $i$ tends to infinity.

Since we can take $G$ to be the complete graph $K_{k+1}$ on $k+1$
vertices, our theorems can be viewed as generalizations of classical
results on the existence of long paths in $G(n,p)$. In particular,
Theorem~\ref{thm:linearpath} generalizes the result of Ajtai,
Koml\'os, and Szemer\'edi and of Fernandez de la Vega, and
Theorem~\ref{thm:hamiltonpath} generalizes the result of Bollob\'as,
and of Koml\'os and Szemer\'edi.

Our theorem can also be placed in a slightly different context.
Recently there has been a number of papers revisiting classical
extremal graph theoretical results of the type `if a graph $G$
satisfies certain condition, then it has some property
$\mathcal{P}$', by asking the following question: ``How strongly
does $G$ possess $\mathcal{P}$?". In other words, one attempts to
measure the robustness of $G$ with respect to the property
$\mathcal{P}$. For example, call a graph on $n$ vertices a
\emph{Dirac graph}, if it has minimum degree at least $\frac{n}{2}$.
Consider the above mentioned theorem which asserts that all Dirac
graphs are Hamiltonian. There are several possible ways one can
measure the robustness of this theorem. Cuckler and Kahn
\cite{CuKa}, confirming a conjecture of S\'ark\"ozy, Selkow, and
Szemer\'edi \cite{SaSeSz}, measured the robustness by counting the
minimum number of Hamilton cycles in Dirac graphs and proved that
all Dirac graphs contain at least $\frac{n!}{(2+o(1))^{n}}$ Hamilton
cycles. In a recent paper \cite{KrLeSu}, we measured the robustness
by taking random subgraphs of Dirac graphs and proved that for every
Dirac graph $G$ on $n$ vertices and $p \gg \frac{\log n}{n}$, a
random subgraph $G_p$ is a.a.s.~Hamiltonian. In the same paper, we
also discussed an alternative measure of robustness where one
analyzes the biased Maker-Breaker Hamiltonicity game on the Dirac
graph. The concept of \emph{resilience} of graphs is another
framework which allows one to measure robustness of graphs. See,
e.g., the paper of Sudakov and Vu \cite{SuVu} for more details. Note
that Theorem \ref{thm:hamiltonpath} in fact measures the robustness
of graphs of minimum degree at least $k$ with respect to containing
paths of length $k$, by taking random subgraphs.

We can also a.a.s. find long cycles in random subgraphs of graphs
with large minimum degree.
\begin{thm}
\label{thm:linearcycle}
Let $\omega$ be a function tending to infinity with $k$
and let $\varepsilon$ be a fixed positive real.
For a finite graph $G$ of minimum degree at least $k$ and 
$p \ge \frac{\omega}{k}$, $G_{p}$ a.a.s. contains a cycle 
of length at least $(1-\varepsilon)k$.
\end{thm}
Similarly to above, Theorem \ref{thm:linearcycle} can be considered
as a generalization of Bollob\'as, Fenner, and Frieze's result. Also
note that this theorem implies a weak form of Theorem
\ref{thm:linearpath}. The proof of this theorem is much more
involved compared to the two previous theorems.

The main technique we use in proving our theorems is a technique
recently developed in \cite{BeKrSu,KrSu}, based on the depth first
search algorithm. In Section \ref{sec:Preliminaries}, we discuss
this technique in detail and also provide some probabilistic tools
that we will need later. Using these tools, in Section
\ref{sec:paths} we prove Theorems \ref{thm:linearpath} and
\ref{thm:hamiltonpath}. Then in Section \ref{sec:cycles} we prove
Theorem \ref{thm:linearcycle}.

\medskip{}

\noindent \textbf{Notation}. A graph $G=(V,E)$ is given by a pair
of its vertex set $V=V(G)$ and edge set $E=E(G)$. We use $|G|$
or $|V|$ to denote the order of the graph. For a subset $X$ of vertices,
we use $e(X)$ to denote the number of edges spanned by $X$, and for
two sets $X,Y$, we use $e(X,Y)$ to denote the number of pairs
$(x,y)$ such that $x\in X,\, y\in Y$, for which $\{x,y\}$ is
an edge (note that $e(X,X)=2e(X)$). $G[X]$
denotes the subgraph of $G$ induced by a subset of vertices $X$.
We use $N(X)$ to denote the collection of vertices which are adjacent
to some vertex of $X$ (we do not require $X$ and $N(X)$ to be
disjoint). For two graphs $G_{1}$ and $G_{2}$ over
the same vertex set $V$, we define their intersection as $G_{1}\cap G_{2}=(V,E(G_{1})\cap E(G_{2}))$,
their union as $G_{1}\cup G_{2}=(V,E(G_{1})\cup E(G_{2}))$, and their
difference as $G_{1}\setminus G_{2}=(V,E(G_{1})\setminus E(G_{2}))$.
Moreover, we let $G\setminus X$ be the induced subgraph $G[V\setminus X]$.

When there are several graphs under consideration, to avoid
ambiguity, we use subscripts such as \emph{$N_{G}(X)$} to indicate
the graph that we are currently interested in. We also use
subscripts with asymptotic notations to indicate dependency. For
example, $O_{\varepsilon}$ will be used to indicate that the hidden
constant depends on $\varepsilon$. To simplify the presentation, we
often omit floor and ceiling signs whenever these are not crucial
and make no attempts to optimize absolute constants involved. We
also assume that the parameter $k$ (which will denote the minimum
degree of the graph under consideration) tends to infinity and
therefore is sufficiently large whenever necessary. All logarithms
will be in base \emph{$e\approx2.718$}.

\section{Preliminaries\label{sec:Preliminaries}}

\subsection{Depth first search algorithm}

Our argument will utilize repeatedly the notion of the depth first
search algorithm. This is a well known graph exploration algorithm,
and we briefly describe it in this section.

The DFS (standing for Depth First Search) algorithm is a graph
search algorithm that visits all vertices of a graph $G=(V,E)$ as
follows. It maintains three sets of vertices, where $S$ is the set
of vertices whose exploration is complete, $T$ is the set of
unvisited vertices, and $U=V\setminus(S\cup T)$. The vertices of $U$
are kept in a stack (the last in, first out data structure). These
three sets will be updated as the algorithm proceeds. We assume that
some order $\sigma$ on the vertices of $G$ is fixed, and the
algorithm prioritizes vertices according to $\sigma$. The algorithm
starts with $S=U=\emptyset$ and $T=V$, and runs until $U\cup
T=\emptyset$. At each round of the algorithm, if the set $U$ is
non-empty, the algorithm queries $T$ for neighbors of the last
vertex $v$ that has been added to $U$, scanning $T$ according to
$\sigma$. If $v$ has a neighbor $u$ in $T$, the algorithm deletes
$u$ from $T$ and inserts it into $U$. If $v$ does not have a
neighbor in $T$, then $v$ is popped out of $U$ and is moved to $S$.
If $U$ is empty, the algorithm chooses the first vertex of $T$
according to $\sigma$, deletes it from $T$ and pushes it into $U$.

Observe that at the time we reach $U\cup T=\emptyset$, we obtain
a rooted spanning forest of our graph (the root of each tree is the
first vertex added to it). At this stage,
in order to complete the exploration of the graph, we make the algorithm
to query all remaining pairs of vertices in $S=V$, not queried before,
in an arbitrary fixed order.

The following properties of the DFS algorithm will be relevant to
us:
\begin{itemize}
  \setlength{\itemsep}{1pt} \setlength{\parskip}{0pt}
  \setlength{\parsep}{0pt}
\item if $T \neq \emptyset$, then every positively answered query increases the size of $S\cup U$ by
one (however, note that having
$h$ positive queries will only guarantee that $|S\cup U|\ge h$,
not $|S\cup U|=h$, since $|S \cup U|$ can also increase at a step where
the stack $U$ is empty);
\item the set $U$ always spans a path (indeed, when a vertex $u$ is added
to $U$, it happens so because $u$ is a neighbor of the last vertex
$v$ in $U$; thus, $u$ augments the path spanned by $U$, of which
$v$ is the last vertex);
\item at any stage, $G$ has no edges between the current
set $S$ and the current set $T$;
\item for every edge $\{v,w\}$ of the graph, there exists a tree component
in the forest produced by the DFS algorithm, in which
$v$ lies on the path from $w$ to the root of the component, or vice versa.
\end{itemize}
In this paper, we utilize the DFS algorithm on random graphs, and
will expose an edge only at the moment at which the existence of it
is queried by the algorithm. More precisely, given a graph $G$ and a
real $p\in[0,1]$, fix an order $\sigma$ to be an arbitrary
permutation, and assume that there is an underlying sequence
$\overline{X}=(X_{i})_{i=1}^{e(G)}$ of i.i.d. Bernoulli random
variables with parameter $p$, which we call as the \emph{query
sequence}. The DFS algorithm gets an answer to its $i$-th query,
asking whether some edge of $G$ exists in $G_{p}$ or not, according to
the value of $X_{i}$; thus the query is answered positively if
$X_{i}=1$, and is answered negatively otherwise. Hence, 
the edge in $G$ whose existence will be
examined on the $i$-th query, depends on the outcome of the 
randomized algorithm. Note that the
obtained graph is distributed according to $G_{p}$. Recently, 
Krivelevich and Sudakov \cite{KrSu} successfully used this
idea to give a simple proof that $p = \frac{1}{n}$ is a sharp
threshold for the appearance of a giant component in a random graph.

\subsection{Probabilistic tools}

We will repeatedly use the technique known as \emph{sprinkling}.
Suppose that for some probability $p$, we wish to establish the fact
$G_{p}\in\mathcal{P}$. It is often more convenient to establish this
fact indirectly by choosing $p_{1}$ and $p_{2}$ so that $G_{p}$ and
$G_{p_{1}}\cup G_{p_{2}}$ have the same distribution ($G_{p_{1}}$
and $G_{p_{2}}$ are independent). We then prove that
$G_{p_{1}}\in\mathcal{P}_{1}$ and $G_{p_{2}}\in\mathcal{P}_{2}$ for
some properties $\mathcal{P}_{1}$ and $\mathcal{P}_{2}$ which
together will imply the fact that $G_{p_{1}}\cup
G_{p_{2}}\in\mathcal{P}$. Of course, a similar argument can be
applied when we split the graph into several independent copies of
random subgraphs formed with probabilities
$p_{1},p_{2},\cdots,p_{k}$. Suppose that we have reals $0 \le p,
p_1, \cdots, p_k \le 1$ and $p_1, \cdots, p_k = o(1)$ satisfying
$p=\sum_{i=1}^{k}p_{i}$. Then the probability of a fixed pair of
vertices forming an edge in $G_{p_1} \cup \cdots \cup G_{p_k}$ is $1
- (1-p_1)\cdots(1-p_k) = (1-o(1))p$. Therefore, in this case
$G_{p_{1}}\cup\cdots\cup G_{p_{k}}$ has the same distribution as
$G_{(1-o(1))p}$, and thus for convenience, we consider
$G_{p_{1}}\cup\cdots\cup G_{p_{k}}$ instead of the graph $G_{p}$,
even though the distribution is not exactly the same. Since we are
only interested in monotone properties $\mathcal{P}$, if we have
$G_{(1-o(1))p}\in\mathcal{P}$ a.a.s., then we also have
$G_{p}\in\mathcal{P}$ a.a.s. Moreover, when we are given the values
of $p_{1},p_{2},\ldots$ beforehand, it is useful to expose the
graphs $G_{p_{i}}$ one at a time. By saying that we \emph{sprinkle}
the next round of edges, we suppose that we consider the outcome of
the graph $G_{p_{i}}$ for the first index $i$ for which $G_{p_{i}}$
has not been exposed.

The following two concentration results are the main probabilistic
tools of this paper (see, e.g., \cite{McDiarmid}). The first theorem
is Chernoff's inequality.
\begin{thm}
Let $\lambda\le np$ be a positive real. If $X$ is a binomial random
variable with parameters $n$ and $p$, then
\[
\mathbb{P}\big(|X-np|\geq\lambda\big)\leq2e^{-\lambda^{2}/(3np)}.
\]

\end{thm}
We will also use the following concentration result proved by Hoeffding
\cite[Theorem 2.3]{McDiarmid}.
\begin{thm}
Let $X_{1},\cdots,X_{n}$ be independent random variables, with $0\le X_{k}\le1$
for each $k$. Let $S_{n}=\sum_{k=1}^{n}X_{k}$ and $\mu=\BBE[S_{n}]$.
Then for every positive $\varepsilon$,
\[
\BBP(S_{n}\ge(1+\varepsilon)\mu)\le e^{-\frac{\varepsilon^{2}\mu}{2(1+\varepsilon/3)}}.
\]

\end{thm}

\section{Long Paths\label{sec:paths}}

In this section we prove Theorems \ref{thm:linearpath} and \ref{thm:hamiltonpath}.
Our first theorem is a slightly stronger version of Theorem \ref{thm:linearpath}.
\begin{thm}
\label{thm:linearpath_strong}Let $p=\frac{c}{k}$ for some $c = o(k)$, 
and let $G$ be a graph of minimum degree at least $k$.

\begin{itemize}
  \setlength{\itemsep}{1pt} \setlength{\parskip}{0pt}
  \setlength{\parsep}{0pt}
\item[(i)]  $G_p$ a.a.s. contains a path of length $(1-2c^{-1/2})k$.
\item[(ii)] If $G$ is a bipartite graph, then $G_p$ a.a.s. contains a path of length $(2-6c^{-1/2})k$.
\item[(iii)] If $c$ tends to infinity with $k$, then for a fixed vertex $v$, there 
a.a.s.~exists a path of length $(1-2c^{-1/2})k$ in $G_p$ which starts at vertex $v$.
\end{itemize}
\end{thm}
\begin{proof}If $c<4$, then the conclusions are vacuously true. Thus
we may assume $c\ge4$, and then for $\varepsilon:=c^{-1/2}$, we have
$\varepsilon\le\frac{1}{2}$.

We will apply the DFS algorithm to the random graph $G_p$, as described
in Section \ref{sec:Preliminaries}. Given a vertex $v$, let $\sigma$ be
an arbitrary ordering of the vertices which has $v$ as its first vertex.
Also assume that we have an underlying query sequence $\overline{X}$.

(i) By Chernoff's inequality, after examining the query sequence for
$\frac{k}{p}=\frac{k^{2}}{c}$ rounds, the probability of 
receiving at least $(1-\varepsilon)k$ positive answers is at least
\[ 1 - e^{-\varepsilon^2k^2/(3k)} = 1 - e^{-\varepsilon^2k/3} = 1-o(1). \]
Condition on this event, and we have $|S\cup U|\ge(1-\varepsilon)k$.
Consider the time at which we reach $|S\cup U|=(1-\varepsilon)k$
(since $S=V$ in the end and $|S|$ changes by at most one at each step,
there necessarily exists such a moment).
Note that we asked less than $\frac{k^{2}}{c}$ queries until this stage. Suppose that
at this time our set $U$ is of size at most $(1-2\varepsilon)k$.
Then we have $|S|>\varepsilon k$. Moreover, since the given graph
has minimum degree at least $k$, each vertex in $S$ has at least
$k-|S\cup U|\ge\varepsilon k$ neighbors in $T$ in the graph $G$.
All the edges between $S$ and $T$ must have been queried by the algorithm and given
a negative answer. Therefore, in order to be in a situation as above,
we must have at least $|S|\cdot\varepsilon k>\varepsilon^{2}k^{2}$
negative answers in our query sequence. However, since we asked at
most $\frac{k^{2}}{c}=\varepsilon^{2}k^{2}$ queries in total, this
cannot happen. Thus we conclude that $|U|\ge(1-2\varepsilon)k=(1-2c^{-1/2})k$,
which implies that there exists a path of length at least $(1-2c^{-1/2})k$,
since the vertices in $U$ form a path.

(ii) By Chernoff's inequality, after examining the query sequence
for $\frac{2k}{p}=\frac{2k^{2}}{c}$ rounds, we a.a.s. have $|S\cup U|\ge(2-5\varepsilon)k$.
Condition on this event, and consider the time at which we reach $|S|=\varepsilon k$.
If $|U|>(2-6\varepsilon)k$ at that point,
then the vertices in $U$ form a path of length at least $(2-6\varepsilon)k$.
Thus assume that $|U|\le(2-6\varepsilon)k$. Since $|S\cup U|\le(2-5\varepsilon)k$,
we examined the entries $X_i$ of the sequence $\overline{X}$, only for
indices $i \le \frac{2k^2}{c}$.
Moreover, since the given graph is bipartite, has minimum degree at
least $k$, and $U$ is a path, each vertex in $S$ has at least 
$k-|S|-\frac{1}{2}|U|\ge 2\varepsilon k$
neighbors in $T$ in $G$. Therefore, we must have at least 
$|S|\cdot2\varepsilon k>2\varepsilon^{2}k^{2}$
negative answers in our query sequence so far. However, since we asked at
most $\frac{2k^{2}}{c}=2\varepsilon^{2}k^{2}$ queries, this cannot
happen. Therefore, $G_{p}$ a.a.s. contains a path of length $(2-6c^{-1/2})k$.

(iii) Let $B$ be the event that there are less than $(1-\varepsilon)k$
positive answers among the first $\frac{k}{p}$ rounds of the query
sequence. By Chernoff's inequality, we have 
$\BBP(B)\le2e^{-\varepsilon^{2}k/3} = o(1)$.
As we have seen in the proof of (i), if $B$ does not hold, then
there exists a path of length at least $(1-2\varepsilon)k$.
In order to compute the probability that there is a path of length $(1-2\varepsilon)k$
starting at $v$, we will bound the probability of the event that
$U\neq\emptyset$ during all the steps involved in reaching $|S\cup U|=(1-\varepsilon)k$,
since if this is the case, then the path of length $(1-2\varepsilon)k$ that
we found above necessarily starts at $v$ (recall that $v$ is the first
vertex in $\sigma$).
Let $A_{i}$ be the event that $U=\emptyset$ at the time we reach
$|S\cup U|=i$, thus we necessarily have $|S|=i$ if this event occurs.
Since $i\le(1-\varepsilon)k$, each vertex in $S$ has at least $\varepsilon k$
neighbors in $T$ in $G$ at that moment. Therefore, when $A_{i}$ occurred, we
received at most $i$ positive answers and at least $i\cdot\varepsilon k$
negative answers to our queries. Thus we can bound the probability
that $A_{i}$ occurs from above by the probability of the event that
there are at most $i$ positive answers among the first $i\cdot\varepsilon k$
queries. Hence by Chernoff's inequality with
$\lambda = i\varepsilon c/2 \geq i$, we have 
$\BBP(A_{i})\le2e^{-(i\varepsilon c)^{2}/(12i\varepsilon c)}=2e^{-ic^{1/2}/12}$.

By the union bound, we get
\begin{align*}
\BBP\Big(B\cup\bigcup_{i=1}^{(1-\varepsilon)k}A_{i}\Big) & \le \BBP(B)+\sum_{i=1}^{(1-\varepsilon)k}\BBP(A_{i})
 = o(1) +\sum_{i=1}^{(1-\varepsilon)k}2(e^{-c^{1/2}/12})^{i}\\
 & \le \frac{2}{e^{c^{1/2}/12}-1}+o(1) = o(1),
\end{align*}
since $c$ tends to infinity with $k$. 
\end{proof}

The second and the third parts of our previous theorem turn out to be
useful in proving Theorem \ref{thm:hamiltonpath}.

Given a graph $G$, consider a path $P=(v_0, v_1, \cdots, v_\ell)$ of length $\ell$
in $G$. Suppose we wish to find a path longer than $P$ in $G$. This could
immediately be done if there exists an edge $\{v_\ell, x\}$ for
some $x \notin V(P)$. P\'osa noticed that an edge of the form $\{v_\ell, v_i\}$
can also be useful, since if such an edge is present in the graph, then
we have a path $P'=(v_0, \cdots, v_i, v_\ell, v_{\ell-1}, \cdots, v_{i+1})$
of length $\ell$ in our graph. Therefore, now we also can find a path of length
greater than $\ell$ if there exists an edge of the form $\{v_{i+1}, x\}$ for
some $x \notin V(P)$. P\'osa's rotation-extension technique is employed
by repeatedly `rotating' the path until we can 'extend' it.

We first state
a special case of Theorem \ref{thm:hamiltonpath}, which can be handled
by using P\'osa's rotation-extension technique. Since the proof is quite
standard, we defer it to later.
\begin{thm}
\label{thm:rotext}There exists a positive real
$\varepsilon_0$ such that following holds for every fixed positive real
$\varepsilon \le \varepsilon_0$.
Let $G$ be a graph on $n$ vertices of minimum degree at least $(1-\varepsilon)k$,
and assume that $n\le(1+\varepsilon)k$. For $p \ge \frac{(1+4\varepsilon)\log k}{k}$,
a random subgraph $G_{p}$ is Hamiltonian a.a.s.
\end{thm}

\begin{proof}[Proof of Theorem \ref{thm:hamiltonpath}]
We may assume that $\varepsilon$ is given so that
$\varepsilon \le \min\{\frac{\varepsilon_0}{48}, \frac{1}{67600}\}$
where $\varepsilon_0$ is given in Theorem \ref{thm:rotext}, since the
conclusion for larger values of $\varepsilon$ follows immediately by monotonicity.
Set $p_{1}=p_{2}=p_{3}=\frac{\varepsilon\log k}{k}$,
$p_{4}=\frac{(1+52\varepsilon)\log k}{k}$, and $p=\frac{(1+55\varepsilon)\log k}{k}$.
We will show that $G_{p_{1}}\cup G_{p_{2}}\cup G_{p_{3}}\cup G_{p_{4}}$
a.a.s. contains a path of length $k$. This will in turn imply that
$G_{p}$ a.a.s. contains a path of length $k$ as discussed in Section
\ref{sec:Preliminaries}.

Given a graph $G$ of minimum
degree at least $k$, by Theorem \ref{thm:linearpath}, we know that
$G_{p_{1}}$ a.a.s. contains a path $P$ of length $\ell=(1-\varepsilon)k$.
Let $X\subset V(G)\setminus V(P)$ be the set of vertices outside $P$ which have
at least $(1-10\varepsilon)|P|$ neighbors in $V(P)$. We consider
two cases depending on the size of $X$.

\smallskip{}

\medskip \noindent \textbf{Case 1. $|X| \ge 2\varepsilon k$.}

\smallskip{}

Redefine $X$ as an arbitrary subset of itself of size exactly $2\varepsilon|P|\le2\varepsilon k$.
Partition the path $P$ into $\frac{1}{2\varepsilon}$ intervals
$P_{1},\cdots,P_{1/2\varepsilon}$, each of length $2\varepsilon|P|$.
By the averaging argument, one can see that there exists an interval
$P_{i}$ for which $e(X,P_{i})\ge(1-10\varepsilon)|X||P_{i}|$. Consider
the bipartite graph $\Gamma$ induced by the two parts $X$ and $P_{i}$,
and note that the number of non-adjacent pairs is at most $10\varepsilon|X||P_{i}|$
(also note that $|X|=|P_{i}|$). Repeatedly remove vertices from $\Gamma$
which have degree at most $(1-8\varepsilon^{1/2})|X|$.
As long as the total number of removed vertices is at most $4\varepsilon^{1/2}|X|$,
each such deletion accounts for at least $4\varepsilon^{1/2}|X|$ non-adjacent
pairs of $\Gamma$.
Thus if we continued the removal for at least $4\varepsilon^{1/2}|X|$ steps, then
by counting the number of non-adjacent pairs in $\Gamma$ in two ways, we have
$4\varepsilon^{1/2}|X| \cdot 4\varepsilon^{1/2}|X| \le10\varepsilon|X||P_{i}|$, which is a contradiction. Hence, the deletion process stops at some step, and we
obtain a subgraph $\Gamma_1$ of $\Gamma$ of minimum degree
at least $(1-8\varepsilon^{1/2})|X|$.

Let $P_{i,0}$ (and $P_{i,1}$) be the leftmost (and rightmost) $9\varepsilon^{1/2}|P_{i}|$
vertices of the interval $P_{i}$. Even after removing the vertices
$P_{i,0},P_{i,1}$ from $\Gamma_{1}$, we are left with a graph $\Gamma_{2}$
of minimum degree at least
\[
(1-8\varepsilon^{1/2})|X|-18\varepsilon^{1/2}|P_{i}|=(1-26\varepsilon^{1/2})|X|\ge \frac{9}{10}|X|.
\]
By Theorem \ref{thm:linearpath_strong} (ii), since $\Gamma_{2}$
is a bipartite graph, in $(\Gamma_{2})_{p_{2}}\subset G_{p_{2}}$,
we can find a path of length at least $2(\frac{9}{10}-o(1))|X| \ge \frac{5}{3}|X|$. By removing
at most two vertices, we may assume that the two endpoints $x$ and
$y$ of this path are both in $X$. Since $\Gamma_{1}$ has minimum
degree at least $(1-8\varepsilon^{1/2})|P_{i}|$, both of these endpoints
have at least $\varepsilon^{1/2}|P_{i}|\ge \varepsilon^{3/2}k$ neighbors
in the sets $P_{i,0}$ and $P_{i,1}$. By Chernoff's inequality, in
the graph $(\Gamma_{1})_{p_{3}}\subset G_{p_{3}}$, we a.a.s. can
find edges of the form $\{x,x_{0}\}$ and $\{y,y_{1}\}$ for $x_{0}\in P_{i,0}$
and $y_{1}\in P_{i,1}$. Thus in the graph $G_{p_{2}}\cup G_{p_{3}}$,
we found a path of length at least $\frac{5}{3}|X|$ which starts
at $x_{0}$, ends at $y_{1}$, and uses only vertices from $X\cup(P_i\setminus(P_{i,0}\cup P_{i,1}))$
as internal vertices. Together with the path $P$, this gives a path
of length at least
\[
|P|-|P_{i}|+\frac{5}{3}|X|\ge|P|+\frac{2}{3}|X|=(1+\frac{1}{3}\varepsilon)k,
\]
 in $G_{p_{1}}\cup G_{p_{2}}\cup G_{p_{3}}$. \smallskip{}

\noindent \textbf{Case 2. $|X| < 2\varepsilon k$.}

\smallskip{}

Let $P=(v_{0},v_{1},\cdots,v_{\ell})$. Let $A_{0}=X\cup V(P)$ and
$B_{0}=V\setminus A_{0}$. Note that $|A_{0}|<(1+\varepsilon)k$,
and $G[B_{0}]$ has minimum degree at least $10\varepsilon|P|-|X|\ge7\varepsilon k$.

If $v_{\ell}$ has at least $\frac{\varepsilon k}{2}$ neighbors in
$B_{0}$ in $G$, then by Chernoff's inequality, in $G_{p_{2}}$,
we a.a.s. can find an edge $\{v_{\ell},w\}$ for some $w\in B_0$. Afterwards,
by Theorem \ref{thm:linearpath_strong} (ii), we a.a.s.~can find a path of
length at least $5\varepsilon k$ in $G_{p_3}[B_{0}]$ starting
at $w$. Together with $P$, these will form a path of length $\ell+1+5\varepsilon k\ge(1+4\varepsilon)k$.
Thus we may assume that $v_{\ell}$ has at least $(1-\frac{\varepsilon}{2})k$
neighbors in $A_{0}$.

Let $Y\subset A_{0}$ be the set of vertices which have at most $(1-10\varepsilon)|P|$
neighbors in $A_{0}$ in $G$. Note that the vertices in $Y$ have
at least $k - (1 - 10\varepsilon)|P| \ge 10\varepsilon k$ neighbors in $B_{0}$. Moreover, by the
definition of the set $X$, all vertices of $Y$ belong to $V(P)$.
Suppose that $|Y|\ge2\varepsilon k$. Then since $|A_{0}|\le(1+\varepsilon)k$,
there are at least
\[ \left(1-\frac{\varepsilon}{2}\right)k - (|A_0| - |Y|) \ge \frac{\varepsilon}{2}k \]
edges in $G$ of the form $\{v_{\ell},v_{i-1}\}$
where $v_{i}\in Y$, and a.a.s.~in $G_{p_{2}}$ we can find one such
edge $\{v_{\ell},v_{i-1}\}$. Afterwards, since $v_{i}$ has at least
$10\varepsilon k$ neighbors in $B_{0}$ in $G$, we a.a.s.~can
find an edge $\{v_{i},w\}$ in $G_{p_2}$
for some $w\in B_{0}$. By Theorem \ref{thm:linearpath_strong} (iii),
there a.a.s. exists a path $P'$ of length at least $5\varepsilon k$
starting at $w$ in $G_{p_{3}}$. The paths $P$ and $P'$ together
with the edges $\{v_{\ell},v_{i-1}\}$ and $\{v_{i},w\}$ will give
a path of length at least $(1+4\varepsilon)k$.

If $|Y|<2\varepsilon k$, then let $A_{1}=A_{0}\setminus Y$ and let
$B_{1}=V\setminus A_{1}$. Note that $|A_{1}|\ge(1-3\varepsilon)k$,
and $G[A_{1}]$ has minimum degree at least $(1-10\varepsilon)|P| - 2\varepsilon k \ge (1-13\varepsilon)k$.
If $k+1\le|A_{1}|<(1+\varepsilon)k$, then we use Theorem \ref{thm:rotext}
to find a.a.s. a path of length $k$ inside $A_1$ in $G_{p_4}$.
Finally, if $|A_{1}|\le k$,
then since $G$ has minimum degree at least $k$, we have $e(A_{1},B_{1})\ge|A_{1}|(k+1-|A_{1}|)\ge k$.
Thus in $G_{p_{2}}$, we can find a.a.s.~an edge $\{v,w\}$ such that $v\in A_{1}$
and $w\in B_{1}$. If $w\in B_{0}$, then Theorem \ref{thm:rotext} a.a.s.~gives
a path of length at least $|A_1| -1 \ge (1-3\varepsilon)k - 1 \ge (1-4\varepsilon)k$ in $G[A_1]_{p_4}$ starting at $v$, and Theorem \ref{thm:linearpath_strong} (iii)
a.a.s. gives a path of length at least
$6\varepsilon k$ in $G[B_{0}]_{p_{3}}$ starting at $w$. These two
paths together with the edge $\{v,w\}$
will give a path of length at least $(1+2\varepsilon)k$. Finally,
if $w\in B_{1}\setminus B_{0}$, then $w$ contains at least
$10\varepsilon k-|Y|\ge8\varepsilon k$
neighbors in the set $B_0$. Therefore in $G_{p_{2}}$, we a.a.s.~can find
an edge $\{w,w'\}$ such that $w'\in B_{0}$. Afterwards, we can proceed
as in the previous case to finish the proof.
\end{proof}

We conclude the section with the proof of Theorem \ref{thm:rotext}.

\begin{proof}[Proof of Theorem \ref{thm:rotext}]
Let $\varepsilon_0 \le \frac{1}{20}$ be a small enough constant,
and let $s = \frac{k}{(\log k)^{3/4}}$. Let $p=\frac{(1+3\varepsilon)\log k}{k}$
and $p_{1}=\cdots=p_{s}=\frac{(\log k)^{5/4}}{k^{2}}$.
We will prove that $G_{p}\cup G_{p_{1}}\cup\cdots\cup G_{p_{s}}$
is a.a.s. Hamiltonian. Since
\[ p_{1}+\cdots+p_{s}=s\cdot\frac{(\log k)^{5/4}}{k^{2}}=\frac{(\log k)^{1/2}}{k}, \]
this will imply that $G_{(1+4\varepsilon)\log k/k}$ a.a.s. contains
a Hamilton cycle.

We first claim that $G_{p}$ a.a.s. satisfies the following properties.
\begin{itemize}
  \setlength{\itemsep}{1pt} \setlength{\parskip}{0pt}
  \setlength{\parsep}{0pt}
\item[1.] for every subset $A$ of vertices of size $|A|\le\frac{n}{(\log k)^{3/2}}$,
we have $|N_{G_{p}}(A)|\ge\varepsilon^3|A|\cdot\log k$,
\item[2.] for every subset $A$ of vertices of size $|A|\ge\frac{n}{(\log k)^{3/4}}$,
we have $|N_{G_{p}}(A)|\ge(1-4\varepsilon)k$, and
\item[3.] $G_{p}$ is connected.
\end{itemize}
We prove these claims through proving that $G_{p}$ a.a.s. has
the following properties:
\begin{itemize}
  \setlength{\itemsep}{1pt} \setlength{\parskip}{0pt}
  \setlength{\parsep}{0pt}
\item[(a)] minimum degree is at least $\varepsilon^2\log k$,
\item[(b)] for all pairs of sets $A$ and $B$ of sizes $|A|\le\frac{n}{(\log k)^{3/2}}$ and $|B|\le\varepsilon^{3}|A|\cdot\log k$, have $e_{G_{p}}(A,B)<\frac{\varepsilon^2}{2}|A|\cdot\log k$, and
\item[(c)] for all pairs of sets $A$ and $B$ of sizes $|A|\ge\frac{n}{(\log k)^{3/4}}$ and $|B|\ge 3\varepsilon k$, have $e_{G_p}(A,B)>0$.
\end{itemize}
For a fixed vertex $v$, the probability $v$
has degree less than $\varepsilon^2\log k$ in the graph $G_{p}$ is
\begin{eqnarray*}
\sum_{i=0}^{\varepsilon^2\log k}{\deg_{G}(v) \choose i}p^{i}(1-p)^{\deg_{G}(v)-i} & \le & \sum_{i=0}^{\varepsilon^2\log k}{(1-\varepsilon)k \choose i}p^{i}(1-p)^{(1-\varepsilon)k-i}\\
 & \le & \sum_{i=0}^{\varepsilon^2\log k}\Big(\frac{e(1-\varepsilon)k}{i}\cdot\frac{p}{1-p}\Big)^{i}\cdot(1-p)^{(1-\varepsilon)k}\\
 & \le & \sum_{i=0}^{\varepsilon^2\log k}\Big(\frac{e(1+3\varepsilon)\log k}{i}\Big)^{i}\cdot e^{-(1+\varepsilon)\log k}\\
 & \le & (\varepsilon^2\log k+1)\cdot\Big(\frac{e(1+3\varepsilon)}{\varepsilon^2}\Big)^{\varepsilon^2\log k}e^{-(1+\varepsilon)\log k},
\end{eqnarray*}
which is $o(k^{-1})$ given that $\varepsilon$ is small enough. By
taking the union bound over all $n \le (1+\varepsilon)k$ vertices, we can deduce (a).
For (b), let $A$ be a set of size $t\le\frac{n}{(\log k)^{3/2}}$
and let $B$ be a set of size $\varepsilon^{3}t\log k$.
If $e_{G_{p}}(A,B)\ge\frac{\varepsilon^2}{2}t\log k$
and $A$ and $B$ are not disjoint, then let $A'=A\setminus B$, $B'=B\setminus A$,
and add the vertices in $A\cap B$, independently and uniformly at
random to $A'$ or $B'$. By linearity of expectation, we have 
\[ \mathbb{E}[e_{G_p}(A',B')] \ge  \frac{1}{4}e_{G_p}(A,B)\ge\frac{\varepsilon^2}{8}t\log k. \]
Hence there exists a choice of disjoint sets $A'\subset A$
and $B'\subset N(A)$ such that $e_{G_p}(A',B')\ge \frac{\varepsilon^2}{8}t\log k$.
Therefore it suffices to show that for every pair of disjoint sets $A$ and $B$
satisfying the bound on the sizes given in (b), we have $e_{G_{p}}(A,B)<\frac{\varepsilon^2}{8}t\log k$.
The probability that $e_{G_{p}}(A,B)\ge\frac{\varepsilon^2}{8}t\log k$
is at most
\[
{|A||B| \choose \varepsilon^2 t\log k/8}\cdot p^{\varepsilon^2 t\log k/8}
\le \left( \frac{e\varepsilon^3t^2 \log k}{\varepsilon^2 t \log k /8} \right)^{\varepsilon^2 t\log k/8} \cdot p^{\varepsilon^2 t\log k/8}
= \Big(8e\varepsilon tp\Big)^{\varepsilon^2 t\log k/8}.
\]
By taking the union bound over all possible sets, we see that
the probability of having a pair of sets violating (b) is at most
\begin{eqnarray*}
\sum_{t=1}^{n/(\log k)^{3/2}}{n \choose \varepsilon^{3}t\log k}^{2}\cdot\Big(8e\varepsilon tp\Big)^{\varepsilon^2 t\log k/8} 
	\le \sum_{t=1}^{n/(\log k)^{3/2}}\Big(\Big(\frac{en}{\varepsilon^{3}t\log k}\Big)^{16\varepsilon}\cdot\frac{16e\varepsilon t\log k}{k} \Big)^{\varepsilon^2 t\log k/8}.
\end{eqnarray*}
By considering all the variables other than $t$ as constant, the
logarithm of the summand on the right hand side can be expressed as
$at\log t+bt$ for some reals $a>0$ and $b$ (given that $8\varepsilon < 1$). Since the second derivative
of this function is positive, the maximum of the summand occurs either
at $t=1$ or $t=n/(\log k)^{3/2}$. From this, one can deduce that the
summand is always $o(n^{-1})$ and the sum is $o(1)$.
For (c), first consider a fixed pair of sets $A$ and $B$
of sizes $|A|\ge\frac{n}{(\log k)^{3/4}}$ and $|B|\ge3\varepsilon k$.
Since the number of vertices
of the graph is $n\le(1+\varepsilon)k$ and minimum degree is at
least $(1-\varepsilon)k$, we have $e_{G}(A,B)\ge\frac{1}{2}|A|\cdot\varepsilon k$.
Therefore, $\BBE[e_{G_p}(A,B)]\ge\frac{\varepsilon}{2}n(\log k)^{1/4}$, and
by Chernoff's inequality, the probability that $e_{G_p}(A,B)=0$ is
at most $e^{-\Omega(\varepsilon n(\log k)^{1/4})}$. Since the number of pairs
of sets $(A,B)$ is at most $2^{2n}$, we can take the union bound over
all choices of $A$ and $B$ to see that (c) hold.

Condition on the event that (a), (b), and (c) holds. Then for a set
$A$ of size at most $|A|\le\frac{n}{(\log k)^{3/2}}$, note that
by (a) we have $e_{G_p}(A,N_{G_p}(A))\ge |A|\varepsilon^2\log k$. Then
by (b), we have $|N_{G_p}(A)|>\varepsilon^{3}|A|\log k$. Therefore we have
Property 1. For Property 2, let $A$ be a set of size at least $\frac{n}{(\log k)^{3/4}}$.
If $|N_{G_p}(A)|<(1-4\varepsilon)k$, then there are no edges between $A$
and $B=V\setminus N_{G_p}(A)$ (recall that $A$ and $N_{G_p}(A)$ are
not necessarily disjoint), where $|B|\ge n-(1-4\varepsilon)k\ge3\varepsilon k$.
This contradicts (c) and cannot happen. Thus we have Property 2 as
well. Property 3 follows from Properties 1 and 2, since they imply that
all connected components are of order at least $(1-4\varepsilon)k > \frac{1}{2}n$.

Condition on the event that $G_{p}$ satisfies Properties 1, 2, and
3 given above. We claim that $G_{p}$ contains a path of length at
least $n-\frac{k}{(\log k)^{3/4}}$, and for all $i\le\frac{k}{(\log k)^{3/4}}$,
conditioned on the event that a longest path in $G_{p}\cup G_{p_1}\cup\cdots\cup G_{p_{i-1}}$ is of length $\ell_i$,
$G_{p}\cup G_{p_{1}}\cup\cdots\cup G_{p_{i}}$ contains a cycle of
length $\ell_i+1$ with probability at least $1-o(k^{-1})$.
Since $G_{p}$ is connected, this will imply
that as long as the graph does not contain a Hamilton path,
the length of a longest path increases by at least one in every
round of sprinkling. Since we start with a cycle of length at least
$n-\frac{k}{(\log k)^{3/4}}$, this will prove that the final graph
is a.a.s. Hamiltonian.

Let $P=(v_{0},\cdots,v_{\ell_i})$ be a longest path in the graph
$G_{i}=G_{p}\cup G_{p_{1}}\cup\cdots\cup G_{p_{i-1}}$ for some $i \ge 1$.
For a set $X=\{v_{a_{1}},\cdots,v_{a_{k}}\}$, we let $X^{-}=\{v_{a_{1}-1},\cdots,v_{a_{k}-1}\}$
and $X^{+}=\{v_{a_{1}+1},\cdots,v_{a_{k}+1}\}$ (if the index becomes
either $0$ or $\ell_i+1$, then we remove the corresponding vertex
from the set). Note that for all sets $X$, we have $X \cup X^+ \cup X^- \subset V(P)$.
Let $X_{0}=\{v_{0}\}$. 

We will iteratively construct
sets $X_{t}$ for $t = 0,1,\cdots$ of size $|X_{t}|\ge(\frac{\varepsilon^3\log k}{4})^{t}$,
as long as $|X_{t}|\le\frac{n}{(\log n)^{3/2}}$, such that $X_{t}\supset X_{t-1}$,
and for every vertex $v\in X_{t}$, there exists a path of length
$\ell_i$ over the vertex set $V(P)$ which starts at $v$ and ends in $v_{\ell_i}$.
Given a set $X_{t-1}$, if $N_{G_{i}}(X_{t-1})\not\subset V(P)$,
then we can find a path of length at least $\ell_i+1$ in $G_{i}$,
which contradicts the assumption on maximality of $P$. Therefore,
$N_{G_{i}}(X_{t-1})\subset V(P)$, and each vertex in $N_{G_{i}}(X_{t-1})\setminus(X_{t-1}\cup X_{t-1}^{-}\cup X_{t-1}^{+})$
gives rise to an `endpoint' from which there exists a path of length
$\ell_i$, and at most two such vertices can give rise to the same endpoint
(see the discussion before the statement of Theorem \ref{thm:rotext}).
Let $X_{t}$ be the union of $X_{t-1}$ with the set of endpoints
obtained as above. We have,
\begin{eqnarray*}
|X_{t}| & \ge & |X_{t-1}|+\frac{1}{2}|N_{G_{i}}(X_{t-1})\setminus(X_{t-1}\cup X_{t-1}^{-}\cup X_{t-1}^{+})|\\
 & \ge & \frac{1}{2}|N_{G_{i}}(X_{t-1})|-\frac{1}{2}|X_{t-1}|\ge\frac{1}{2}|X_{t-1}|\cdot(\varepsilon^3\log k-1)\ge\Big(\frac{\varepsilon^3\log k}{4}\Big)^{t}
\end{eqnarray*}
(where we used Property 1 in the second to last inequality). Repeat
the argument until the first time we reach a set of size $|X_{t}|>\frac{n}{(\log k)^{3/2}}$,
and redefine $X_{t}$ as a subset of size exactly $\frac{n}{(\log k)^{3/2}}$
which contains $X_{t-1}$. By repeating the argument above, we can
find a set $X_{t+1}$ of size at least $\frac{\varepsilon^3\log k}{4}\cdot\frac{n}{(\log k)^{3/2}}>\frac{n}{(\log k)^{3/4}}$.
By repeating the argument once more (now using Property 2 instead
of Property 1), we can find a set $X_{t+2}$ of size at least $\frac{1}{2}(1-5\varepsilon)k$.

For each vertex $v\in X_{t+2}$, there exists a path of length $\ell_i$
which starts at $v$ and uses only vertices from $V(P)$. Thus if
there exists an edge between $X_{t+2}$ and $V\setminus V(P)$ in
$G_{i}$, then we can find a path of length at least $\ell_i+1$ and
contradict maximality of $P$. If $|V\setminus V(P)|\ge\frac{n}{(\log k)^{3/4}}$,
then Property 2 implies the existence of such edge
(since $\frac{1}{2}(1-5\varepsilon)k + (1-4\varepsilon)k > n$). This shows that
we have $\ell_i\ge n-\frac{n}{(\log k)^{3/4}}$ for all $i \ge 1$. In particular,
$G_p$ contains a path of length at least $n-\frac{n}{(\log k)^{3/4}}$.
 For each vertex $v\in X_{r+2}$,
by applying the argument of the previous paragraph to the other endpoint
of the path starting at $v$, we can find a set $Y_{v}$ of size at
least $\frac{1}{2}(1-5\varepsilon)k$ such that for every $w\in Y_{v}$,
there exists a path of length $\ell_i$ which starts at $v$ and ends
at $w$. Since $n \le (1+\varepsilon)k$ and the minimum degree of $G$ is
at least $(1-\varepsilon)k$, there are at least
$\frac{1}{2}\cdot\Big(\frac{1}{2}(1-5\varepsilon)k - 2\varepsilon k\Big)^{2}\ge\frac{1}{16}k^{2}$
pairs such that if some pair appears in $G_{p_{i}}$, then $G_{p_{i}}$
contains a cycle of length $\ell_i+1$. Since $p_{i}=\frac{(\log k)^{5/4}}{k^{2}}$,
by Chernoff's inequality, the probability that such edge appears in
$G_{p_{i}}$ is at least $1-e^{-\Omega((\log k)^{5/4})}=1-o(k^{-1})$.
This proves our claim.\end{proof}

\section{Cycles of length $(1-o(1))k$\label{sec:cycles}}

In this section we prove Theorem \ref{thm:linearcycle}.

\subsection{High connectivity and long cycles}

We start with a simple lemma based on the DFS algorithm which allows
us to claim the a.a.s. existence of a cycle of length linear in
the average degree of the graph.
\begin{lem}
\label{lem:DFS_cycle}Let $\alpha$ be a fixed positive real. Let
$G$ be a graph of average degree $\alpha k$, and let $p=\frac{\omega}{k}$
for some function $\omega=\omega(k)\ll k$ that tends to infinity as $k$
does. Then $G_{p}$ a.a.s.~contains a cycle of length at least $(\frac{1}{2}-\frac{5}{\sqrt{\omega}})\alpha k$.
Moreover, if $G$ is a bipartite graph, then $G_{p}$ a.a.s.~contains
a cycle of length at least $(1-\frac{10}{\sqrt{\omega}})\alpha k$.\end{lem}
\begin{proof}
Let $n$ be the number of vertices of $G$ and set $p_{1}=p_{2}=\frac{\omega}{2k}\ll 1$.
We will show that $G_{p_{1}}\cup G_{p_{2}}$
a.a.s. contains a cycle of length at least $(\frac{1}{2}-\frac{5}{\sqrt{\omega}})\alpha k$.

Consider the DFS algorithm applied to the graph $G_{p_{1}}$ starting
from an arbitrary vertex. By Chernoff's
inequality, after examining the query sequence for $\frac{2n}{p_{1}}$
steps, we a.a.s. receive at least $n$ positive answers. Condition
on this event. At the time at which $T$ becomes an empty set (therefore
when we complete exploring the component structure), we know that since $|S\cup U|=n$,
the length of the query sequence is at most $\frac{2n}{p_{1}}$. The
rooted spanning forest we found induces a partial order on the vertices
of the graph, where for two vertices $x,y$, we have $x<y$ if and
only if $x$ is a predecessor of $y$ in one of the rooted trees in
the spanning forest.

In the DFS algorithm, every edge $\{x,y\}$ of $G$ for which $x$
and $y$ are incomparable in the partial order, must have been queried
and answered negatively. Therefore, there can only be at most $\frac{2n}{p_1}$
such edges. Since the average degree of the graph is $\alpha k$,
this means that there are at least $\frac{\alpha nk}{2}-\frac{2n}{p_1}\ge(\frac{1}{2}-\frac{4}{\alpha\omega})n\cdot\alpha k$
edges $\{x,y\}$ of $G$ for which $x<y$ or $y<x$. Hence, there
exists a vertex $v$ incident to at least $(\frac{1}{2}-\frac{4}{\alpha\omega})\alpha k$
edges $\{w,v\}$ of $G$ for which $w<v$.
By definition, the other endpoint of all these edges
lie on the path from $v$ to the root of the tree that $v$ belongs to.
The probability that one of the edges among the $\frac{\alpha k}{\sqrt{\omega}}$ 
farthest reaching edges appear is 
\[ 1 - (1-p)^{\alpha k / \sqrt{\omega}} \ge 1 - e^{\alpha \sqrt{\omega}} = 1-o(1). \]
Thus there a.a.s.~exists at least one edge among the farthest
reaching $\frac{\alpha k}{\sqrt{\omega}}$ edges. This edge gives a
cycle of length at least $(\frac{1}{2}-\frac{5}{\sqrt{\omega}})\alpha k$.

Moreover, if $G$ is a bipartite graph, then this edge gives a cycle
of length $(1-\frac{10}{\sqrt{\omega}})\alpha k$, since the vertex
$v$ can only be adjacent to every other vertex in the path from $v$
to the root of the tree that contains $v$.
\end{proof}
Let $t$ be a positive integer. A graph $G$ is \emph{$t$-vertex-connected}
(or \emph{$t$-connected} in short) if for every set $S$ of at most
$t-1$ vertices, the graph $G\setminus S$ is connected. Here we state
some facts about highly-connected graphs without proof. The fourth
part is a result of Mader \cite{Mader}, and the fifth part is a result
of Menger \cite{Menger}. We refer readers to Diestel's graph theory
book \cite{Diestel} for more information on highly-connected graphs.

\begin{lem}
\label{lem:highconnectedgraph}Let $t$ be a positive integer, and let $G$, $G'$ be $t$-connected graphs.

\begin{itemize}
  \setlength{\itemsep}{1pt} \setlength{\parskip}{0pt}
  \setlength{\parsep}{0pt}
\item[(i)] $G$ remains connected even after removing a combination of $t-1$ edges and vertices.
\item[(ii)] If $v \notin V(G)$ has at least $t$ neighbors in $G$, then $G \cup \{v\}$ is also $t$-connected.
\item[(iii)] If $|V(G)\cap V(G')| \ge t$, then $G \cup G'$ is also $t$-connected.
\item[(iv)] Every graph of average degree at least $4t$ contains a $t$-connected subgraph.
\item[(v)] For every pair of subsets $A$ and $B$ of $V(G)$, there are 
$\min\{t, |A|, |B|\}$ vertex-disjoint paths in $G$ that connect $A$ and $B$.
\end{itemize}
\end{lem}
The main strategy we use in proving Theorem \ref{thm:linearcycle}
is to find in the random subgraph a highly connected subgraph
that contains many vertex disjoint cycles.
Lemma \ref{lem:DFS_cycle} will be used to find
vertex disjoint cycles. Afterwards the connectivity condition will
allow us to `patch' the cycles into a long cycle of desired length.
This is similar in spirit to a theorem of Locke \cite{Locke} which
asserts that a 3-connected graph with a path of length
$\ell$ contains a cycle of length at least $\frac{2}{5}\ell$.

\begin{lem}
\label{lem:iterative_cycle}Let $\alpha$ be a fixed positive real,
$t$ be a fixed positive integer, and let $p=\frac{\omega}{k}$ for some
function $\omega=\omega(k)\ll k$ that tends to infinity as $k$ does.
Let $G_{1}$ and $G_{2}$ be graphs defined over the same set of $n$
vertices. Suppose that at least $(1-\frac{1}{t})n$ vertices of $G_{1}$
have degree at least $\alpha k$, and that $G_{2}$ is $t$-connected.
Then the graph $(G_{1})_{p}\cup G_{2}$ a.a.s.
contains a cycle of length at least $(1-\frac{10}{t})\alpha k$.
\end{lem}
\begin{proof}
Let $p_{1}=p_{2}=\cdots=p_{t}=\frac{\omega}{tk}$. Suppose that
we have found a cycle of length $\ell\le(1-\frac{10}{t})\alpha k$ after
$i-1$ rounds of sprinkling. We claim that we can then find a.a.s. a cycle
of length at least $(1-\frac{1}{t})\ell+\frac{2\alpha k}{t}$, by
sprinkling one more round with probability $p_i$.
If this is the case, then since
\[
\Big((1-\frac{1}{t})\ell+\frac{2\alpha k}{t}\Big)-\ell=\frac{2\alpha k}{t}-\frac{\ell}{t}\ge\Big(\frac{2}{t}-\frac{1}{t}\Big)\alpha k=\frac{1}{t}\alpha k,
\]
after sprinkling at most $t$ rounds, we will a.a.s. find a cycle of length
at least $(1-\frac{10}{t})\alpha k$.

To prove the claim, suppose that we are given a cycle $C$ of length
$\ell\le(1-\frac{10}{t})\alpha k$. Let $V'=V\setminus V(C)$. Since
$G_{1}$ has $(1-\frac{1}{t})n$ vertices of degree at least $\alpha k$,
the graph $G[V']$ has at least $(1-\frac{1}{t})n-\ell$ vertices
which have degree at least $\alpha k-\ell\ge\frac{10}{t}\alpha k$.
Therefore the average degree of $G[V']$ is at least
\[
\frac{\Big((1-\frac{1}{t})n-\ell\Big)\cdot\frac{10}{t}\alpha k}{n-\ell}=\Big(1-\frac{1}{t}-\frac{\ell}{t(n-\ell)}\Big)\cdot\frac{10}{t}\alpha k=\Big(1-\frac{1}{t}-\frac{1}{t(n/\ell-1)}\Big)\cdot\frac{10}{t}\alpha k,
\]
which is minimized when $\ell$ is maximized. Since $\ell\le(1-\frac{10}{t})\alpha k\le(1-\frac{10}{t})n$,
the average degree of $G[V']$ is at least
\[
\Big(1-\frac{1}{t}-\frac{(1-\frac{10}{t})n}{t\cdot\frac{10}{t}n}\Big)\cdot\frac{10}{t}\alpha k=\frac{9}{t}\alpha k.
\]

Thus by Lemma \ref{lem:DFS_cycle}, after sprinkling one more round,
we a.a.s.~can find a cycle $C'$ in $G[V']$ of length at least $(\frac{1}{2}-o(1))\frac{9\alpha k}{t}\ge\frac{4\alpha k}{t}$.
Since $G_{2}$ is a $t$-connected graph, there exist $t$ vertex
disjoint paths that connect $C$ to $C'$ (see Lemma \ref{lem:highconnectedgraph}
(v)). Among these paths, consider the two whose intersection point
with $C$ are closest to each other (along the distance
induced by $C$). Using these two paths to merge $C$ and $C'$, we
a.a.s.~can find a cycle of length at least
\[
\Big(1-\frac{1}{t}\Big)|V(C)|+\frac{1}{2}|V(C')|\ge\Big(1-\frac{1}{t}\Big)\ell+\frac{2\alpha k}{t},
\]
as claimed.
\end{proof}
Our next lemma is similar to the lemma above, but will be applied
under slightly different circumstances.
\begin{lem}
\label{lem:combine_cycle} Suppose that $\ell$ and $t$ are
given integers satisfying $\ell \ge t$.
Let $G$ be a $t$-vertex-connected graph
that contains $s$ vertex-disjoint cycles of lengths at least $\ell$ each.
Then $G$ contains a cycle of length at least \[ \left(1-\frac{s}{t}\right)^{s-1}\ell+\sum_{i=0}^{s-2}\left(1-\frac{s}{t}\right)^{i}\cdot\frac{\ell}{2}.\]

Thus if $t$ is large enough depending on $s$, then
$G$ contains a cycle of length at least $\frac{s}{2}\ell$.
\end{lem}
\begin{proof}
We will find a cycle of desired length by an iterative
process; for $h=1,2,\cdots,s$, after the $h$-th step,
we find a cycle of length at least $\ell_{h}=(1-\frac{s}{t})^{h-1}\ell+\sum_{i=0}^{h-2}\left(1-\frac{s}{t}\right)^{i}\cdot\frac{\ell}{2}$,
and $s-h$ other cycles of length at least $\ell$ each which are all vertex-disjoint.
Note that the statement is vacuously true for $h=1$ by the given
condition. The statement for $h=s$ corresponds to the statment
of the lemma.

Given a cycle $C_{h}$ of length at least $\ell_{h}$ and $s-h$ other vertex
disjoint cycles of length at least $\ell$ each, let $X = V(C_h)$ and
$Y$ be the union of the set of vertices of the cycles of length at least $\ell$.
By the $t$-connectivity of our graph, Lemma \ref{lem:highconnectedgraph} (v), and
the fact 
\[ \min\{|X|, |Y|\} \ge \ell \ge t, \]
we see that there are $t$ vertex-disjoint paths that connect $X$ to $Y$. 
By the pigeonhole principle, at
least $\frac{t}{s-h}$ of these paths connect the cycle of length
at least $\ell_{h}$ to one fixed cycle of length at least $\ell$. Among these paths,
consider the two whose intersection points with $C_{h}$ are closest
to each other. Using these two paths, we can find a cycle of length
at least
\[
\Big(1-\frac{s-h}{t}\Big)\ell_{h}+\frac{\ell}{2}\ge\Big(1-\frac{s}{t}\Big)\ell_{h}+\frac{\ell}{2}=\ell_{h+1}.
\]
Moreover, note that we still have at least $s-h-1$ vertex-disjoint
cycles which are also disjoint to the new cycle we found. Therefore
in the end, after the $s$-th step, we will find our desired cycle.
For the second part, note that
\begin{align*}
\ell_{s} & = \Big(1-\frac{s}{t}\Big)^{s-1}\ell+\sum_{i=0}^{s-2}\left(1-\frac{s}{t}\right)^{i}\cdot\frac{\ell}{2}=\Big(1-\frac{s}{t}\Big)^{s-1}\frac{\ell}{2}+\sum_{i=0}^{s-1}\left(1-\frac{s}{t}\right)^{i}\cdot\frac{\ell}{2}\\
 & = \Big(1-\frac{s}{t}\Big)^{s-1}\frac{\ell}{2}+\frac{t}{s}\left(1-\Big(1-\frac{s}{t}\Big)^{s}\right)\cdot\frac{\ell}{2}.
\end{align*}
If $t$ is large enough depending on $s$, we have $(1-\frac{s}{t})^{s-1}=1-o_{t}(1)$
and $1-(1-\frac{s}{t})^{s}=\frac{s^{2}}{t}-O_{t}(\frac{s^{4}}{t^{2}})$.
Therefore in this case,
\[
\ell_{s}\ge\Big(1-o_{t}(1)\Big)\frac{\ell}{2}+\frac{t}{s}\left(\frac{s^{2}}{t}-O_{t}\Big(\frac{s^{4}}{t^{2}}\Big)\right)\frac{\ell}{2}\ge\frac{s}{2}\ell.
\]
\end{proof}

\subsection{Finding long cycles}

In this subsection we prove Theorem \ref{thm:linearcycle}.

We first state a structural lemma, which a.a.s.~finds almost
vertex-disjoint highly connected subgraphs in the random subgraph
of our given graph.
Afterwards, we will use the lemmas developed in the previous subsection
in order to find a long cycle in various situations.
\begin{lem}
\label{lem:structure}Let $\varepsilon\le\frac{1}{2}$ be a fixed
positive real. Let $G$ be a graph on $n$ vertices
with minimum degree at least $k$,
and let $p=\frac{\omega}{k}$ for some function $\omega=\omega(k) \ll k$
that tends to infinity as $k$ does. Suppose that $G$ does not contain
a bipartite subgraph of average degree at least $\frac{5}{4}k$. Then
$G_{p}$ a.a.s. admits a partition $V=X\cup Y$ of its vertex set,
and contains a collection $\mathcal{C}$ of subgraphs of $G_{p}$
satisfying:
\begin{itemize}
  \setlength{\itemsep}{1pt} \setlength{\parskip}{0pt}
  \setlength{\parsep}{0pt}
\item[(a)] for every $C \in \mathcal{C}$, $C$ is $(\log \omega)^{1/5}$-connected;
\item[(b)] the sets $X \cap V(C)$ for $C \in \mathcal{C}$ form a partition of $X$, and $|Y|=o(n)$;
\item[(c)] for every $C \in \mathcal{C}$, one of the following holds:
\begin{itemize}
\item[(i)] the graph $G[V(C)]$ contains at least $(1-\varepsilon)|V(C)|$ vertices of degree at least $(1-\varepsilon)k$, or
\item[(ii)] $|Y \cap V(C)| = o(|V(C)|)$, the graph $G[X \cap V(C)]$ contains at least $(1-\varepsilon)|V(C)|$ vertices
of degree at least $\frac{k}{8}$, and there exists a bipartite graph $\Gamma_C \subseteq G$ with parts $X \cap V(C)$ and
$Y \setminus V(C)$ which contains at least $\frac{|V(C)|\varepsilon^2 k}{4}$ edges and
has maximum degree at most $\frac{8}{\varepsilon}k$.
\end{itemize}
\end{itemize}
\end{lem}
We defer the proof of the structural lemma to later and first
prove Theorem \ref{thm:linearcycle} using the structural lemma. Let
$G$ be a given graph of minimum degree at least $k$ on $n$ vertices,
and let $\varepsilon$ be a given positive real. For $p=\frac{\omega}{k}$,
it suffices to prove the statement for $\omega \ll k$ since the conclusion
for larger $\omega$ follows from monotonicity.
Set $p_{1}=p_{2}=p_{3}=\frac{\omega}{3k}$.
Suppose that $\varepsilon \le \frac{1}{50}$ is given (for larger values of $\varepsilon$,
we may assume that $\varepsilon=\frac{1}{50}$).

\medskip{}

\noindent \textbf{Case 1. There exists a bipartite subgraph of $G$ of average degree at least $\frac{5}{4}k$.}

\smallskip

We can apply Lemma \ref{lem:DFS_cycle} to the bipartite subgraph
to a.a.s. find a cycle in $G_{p_{1}}$ of length at least $(1-o(1))\frac{5}{4}k\ge k+1$.

If $G$ does not contain such a bipartite subgraph, then we apply
Lemma \ref{lem:structure} to a.a.s. find a collection $\mathcal{C}$
of subgraphs which induce highly-connected subgraphs of $G_{p_{1}}$.

\medskip{}

\noindent \textbf{Case 2. There exists $C \in \mathcal{C}$ such that Property (c)-(i) holds.}

\smallskip

By Lemma \ref{lem:structure} (a), $C$ is $\frac{1}{\varepsilon}$ connected, and
we can apply Lemma \ref{lem:iterative_cycle} with $t = \frac{1}{\varepsilon}$, $\alpha = 1-\varepsilon$, $G_1 = G[V(C)]$ and $G_2 = C$ to a.a.s.~obtain a cycle of length at least $(1-10\varepsilon)\cdot(1-\varepsilon)k\ge(1-11\varepsilon)k$
in $G_{p_2}[V(C)] \cup C \subset G_{p_1} \cup G_{p_2}$.

\medskip{}

\noindent \textbf{Case 3. Property (c)-(ii) holds for all $C \in \mathcal{C}$.}

\smallskip

For each $C \in \mathcal{C}$, there exists a bipartite graph $\Gamma_C$ with
parts $X \cap V(C)$ and $Y \setminus V(C)$ which has at least
$\frac{|V(C)|\varepsilon^2 k}{4}$ edges and maximum degree at most $\frac{8 k}{\varepsilon}$.
Expose the graph $G_{p_{2}}$, and for each $C\in\mathcal{C}$, let
$M_{C}$ be a maximum matching in $(\Gamma_C)_{p_2}$. Let $\mathcal{C}'=\{C\in\mathcal{C}\,:\,|M_{C}|\ge \frac{\varepsilon^{3}|V(C)|}{128} \}$.
\begin{lem}
\label{lem:stick}
We a.a.s. have $\sum_{C\in\mathcal{C}'}|V(C)|\ge\frac{n}{2}$.
\end{lem}
\begin{proof}
For a graph $C\in\mathcal{C}$, we first estimate the probability
that $C\notin\mathcal{C}'$. Let $X_{C}=X\cap V(C)$, $Y_{C}=Y\setminus V(C)$,
and $m_C$ be the number of edges of $\Gamma_C$ (thus $\frac{|V(C)|\varepsilon^2 k}{4} \le m_C \le |V(C)| \cdot \frac{8k}{\varepsilon}$).
Since the maximum degree is at most $\frac{8k}{\varepsilon}$, we know that for
every collection of $t \le \frac{\varepsilon^3 |V(C)|}{128}$ vertex-disjoint edges,
there are at least $\frac{|V(C)|\varepsilon^{2}k}{4}-2t \cdot \frac{8k}{\varepsilon} \ge \frac{\varepsilon^{2}|V(C)|k}{8}$
edges in $\Gamma_C$ not intersecting any of the edges in the collection.
Therefore the probability that $|M_C|=t$ is at most
\begin{align*}
 {m_C \choose t} \cdot p_2^t \cdot (1-p_2)^{\varepsilon^2|V(C)|k/8}
 &\le \left( \frac{em_Cp_2}{t} \right)^t (1-p_2)^{\varepsilon^2|V(C)|k/8} \\
 &\le \left( \frac{e}{t} \cdot \frac{8|V(C)|k}{\varepsilon} \cdot \frac{\omega}{3k} \right)^t e^{-\varepsilon^2|V(C)|kp_2/8} \\
 &= \left( \frac{8e\omega|V(C)|}{3t\varepsilon} \right)^{t} e^{-\varepsilon^2|V(C)|\omega/24}.
\end{align*}
By parameterizing $t$ as $t = \alpha|V(C)|$ ($\alpha \le \frac{\varepsilon^3}{128}$),
the right hand side becomes
\begin{align*}
\left( \frac{8e\omega}{3\alpha \varepsilon} \right)^{\alpha|V(C)|} e^{-\varepsilon^2|V(C)|\omega/24}
= e^{\alpha \log (8e\omega/(3\alpha\varepsilon)) |V(C)|} e^{-\varepsilon^2|V(C)|\omega/24}
\le e^{-\varepsilon^2|V(C)|\omega/48}.
\end{align*}
By taking the union bound over
all values of $t$ from $1$ to $\frac{\varepsilon^3 |V(C)|}{128}$, we see
that the probability of $|M_C| < \frac{\varepsilon^3|V(C)|}{128}$, or equivalently
$C \notin \mathcal{C}$, is at most
\begin{align*}
\frac{\varepsilon^3 |V(C)|}{128} \cdot e^{-\varepsilon^2|V(C)|\omega/48} = o(1).
\end{align*}
By Markov's inequality, it follows that $\sum_{C \notin \mathcal{C}'} |X \cap V(C)| < \frac{n}{4}$
a.a.s. If this event holds, then since $|X \cap V(C)| = (1-o(1))|V(C)|$ for all $C \in \mathcal{C}$,
we have
\begin{align*}
\sum_{C\in \mathcal{C}'} |V(C)| & =  (1+o(1)) \sum_{C\in \mathcal{C}'} |X \cap V(C)|
 =  (1+o(1)) \left(\sum_{C\in \mathcal{C}} |X \cap V(C)| -\sum_{C\notin \mathcal{C}'} |X \cap V(C)| \right) \\
& = (1+o(1)) \left( (1-o(1))n - \frac{n}{4} \right) \ge \frac{n}{2}.
\end{align*}
\end{proof}
Condition on the conclusion of Lemma \ref{lem:stick}. Consider an
auxiliary bipartite graph $\Gamma$ over a vertex set consisting of
two parts $\mathcal{C}$ and $Y$ (where $Y$ is the set given by
Lemma \ref{lem:structure}). A pair $\{C,y\}$ forms an edge if $y$
is an endpoint of some edge in $M_{C}$. Since $|X \cap V(C)| \ge \frac{k}{8}$
and $|X| = (1-o(1))n$, the number of vertices of $\Gamma$
is $|\mathcal{C}|+|Y|\le\frac{n}{k/8}+o(n)=o(n)$ and the
number of edges is $\sum_{C\in\mathcal{C}} |M_C| \ge \sum_{C\in\mathcal{C}'}\frac{\varepsilon^{3}}{128}|V(C)|\ge\frac{\varepsilon^{3}}{256}n$.
Let $t\ge100$ be a large enough constant. By Lemma \ref{lem:highconnectedgraph} (iv),
there exists a $t$-connected subgraph $\Gamma'$ of $\Gamma$ over
the vertex set $C_{1},\cdots,C_{s},y_{1},\cdots,y_{s'}$ of $\Gamma$.
We claim that the induced subgraph $H$ of $G_{p_{1}}\cup G_{p_{2}}$
on the vertex set $V(C_{1})\cup\cdots\cup V(C_{s})\cup\{y_{1},\cdots,y_{s'}\}$
is $t$-connected (note that $s,s'\ge t$). Suppose that this is the
case. Then since the sets $V(C_{1})\cap X,\cdots,V(C_{s})\cap X$
are vertex disjoint and each graph $G[V(C_{i})\cap X]$ contains at
least $(1-\varepsilon)|V(C_{i})|$ vertices of degree at least $\frac{k}{8}$,
by Lemma \ref{lem:iterative_cycle} for each fixed $i$, $G_{p_{3}}[V(C_{i})\cap X]$
a.a.s. contains a cycle of length at least $(1-10\varepsilon)\frac{k}{8}\ge\frac{k}{10}$.
Thus in $H\cup G_{p_{3}}$ we a.a.s. have at least $(1-o(1))s$ vertex disjoint
cycles of length at least $\frac{k}{10}$ in the graph.
Since $H$ is $t$-connected, for large enough $t$, by Lemma \ref{lem:combine_cycle} we can
use $30$ of the vertex disjoint cycles to a.a.s.~find in
$H \cup G_{p_3} \subseteq G_{p_1} \cup G_{p_2} \cup G_{p_{3}}$ a cycle of length at least
$\frac{1}{2}\cdot 30 \cdot\frac{k}{10} > k$.

Therefore to conclude the proof of Theorem \ref{thm:linearcycle}, it suffices
to prove that $H$ is $t$-connected. Let $S$ be a subset
of at most $t-1$ vertices of $V(H)$. It suffices to prove
that $H\setminus S$ is a connected graph. We do this by
exploiting the $t$-connectivity of $\Gamma'$. Let $v$ and $w$
be two arbitrary vertices of $H\setminus S$. For a vertex
$x\in S$, if $x\in Y$, then remove $x$ from $\Gamma'$. Otherwise,
if $x\in X$ and there is a matching edge belonging to some $M_{C}$
incident to $x$, then remove the corresponding edge from $\Gamma'$ (note
that there is a one-to-one correspondence between such edges and edges
of $\Gamma'$). 

First suppose that $v,w\in X$. Since we removed
at most $t-1$ vertices/edges from the graph $\Gamma'$, without loss
of generality, there still exists a path $C_{1}z_{1}C_{2}\cdots z_{h-1}C_{h}$
in $\Gamma'$ for $v\in V(C_{1})$ and $w\in V(C_{h})$. For each
$z_{i}$, there exist vertices $z_{i}'\in V(C_{i})\setminus S$ and
$z_{i}''\in V(C_{i+1})\setminus S$ such that $\{z_{i},z_{i}'\}\in M_{C_{i}}$
and $\{z_{i},z_{i}''\}\in M_{C_{i+1}}$. Let $z_0'' = v$ and $z_h' = w$.
Since each $C_{i}$ is $t$-connected,
for $i=1,2,\cdots,h$, we can find a path from $z_{i-1}''$ to $z_{i}'$
in $C_{i+1}\setminus S$. By combining these 
paths with the edges $\{z_i', z_i\}$ and $\{z_i, z_i''\}$ for $i = 
1,2,\cdots,h-1$, 
we obtain a path from $z_{0}''=v$ to $z_{h}'=w$.

Second, suppose that $v \in X$ and $w \in Y$. Since we removed at most $t-1$
edges from $H$, there exists an edge of $H$ incident to $w$ whose other 
endpoint $w'$ is in $X$. By the case above, we see that there exists a path
from $v$ to $w'$ which implies that there is a path from $v$ to $w$. The last
case when $v,w \in Y$ can be handled similarly. \hfill $\qed$

\subsection{Proof of the structural lemma}

In this subsection, we prove the structural lemma, Lemma \ref{lem:structure}.
The proof is quite technical so we begin this section by briefly explaining
its idea. 

Let $\mathcal{C} = \{C_1, C_2, \cdots\}$ be a collection of 
edge-disjoint $t$-connected 
subgraphs of $G_p$ (where $t$ is some large integer), which covers
the maximum number of edges of $G_1$ and has the minimum number of subgraphs in it.
Note that if the number of vertices of $G$ is $O(k)$, then the collection
$\mathcal{C}$ likely consits of a single subgraph. However, we put no restriction
on the number of vertices, and thus $G_p$ might even consist of several 
connected components. Thus $\mathcal{C}$ is a non-trivial collection of subgraphs
of $G_p$. This collection will have interesting properties which will eventually imply
our lemma.

Note that two subgraphs $C_i$ and $C_j$ can only share at most $t-1$ vertices
since otherwise $C_i$ and $C_j$ can be combined into a single $t$-connected
subgraph of $G_p$ to contradict the minimality
of the collection. Hence every pair subgraphs in $\mathcal{C}$ are `almost' disjoint.
Moreover, if there are too many edges of the graph $G$ 
not covered by any of the subgraphs in $\mathcal{C}$, then we will be able to
find a $t$-connected subgraph of $G_1$ which is edge-disjoint to all subgraphs
in $\mathcal{C}$ and thus contradicts the maximality of the collection. 
Thus most edges of $G$ are covered by some subgraph in $\mathcal{C}$.

Afterwards, we find a subcollection $\mathcal{C}'$ of $\mathcal{C}$ for which
the following holds: the number of vertices which are covered by
at least two subgraphs in $\mathcal{C}'$ is small. Moreover, the collection
$\mathcal{C}'$ will maintain the property that most edges of $G$ are covered
by some subgraph. Now let $X$ be the set of vertices which are covered
by exactly one subgraph in $\mathcal{C}'$, and $Y$ be the rest of the vertices.
We can see that most edges of $G$ lie within subgraphs in the
collection $\mathcal{C}'$, and that every pair of subgraphs in $\mathcal{C}'$ 
are only allowed to intersect in $Y$. This illustrates how the structure
claimed in Lemma \ref{lem:structure} arises from the collection $\mathcal{C}$. 

\medskip

We first prove the following lemma, which forms an intermediate step
in proving Lemma \ref{lem:structure}.
\begin{lem}
\label{lem:structure1}Let $G$ be a graph on $n$ vertices
with minimum degree at least
$k$, and let $p=\frac{\omega}{k}$ for some function $\omega=\omega(k)\ll k$
that tends to infinity as $k$ does. Suppose that $G$ does not contain
a bipartite subgraph of average degree at least $\frac{5}{4}k$. Then
$G_{p}$ a.a.s. admits a partition $V=X\cup Y$ of its vertex set,
and contains a collection $\mathcal{C}$ of subgraphs satisfying
the following:
\begin{itemize}
  \setlength{\itemsep}{1pt} \setlength{\parskip}{0pt}
  \setlength{\parsep}{0pt}
\item[(a)] every graph $C \in \mathcal{C}$ is $(\log \omega)^{1/5}$-connected;
\item[(b)] the sets $X \cap V(C)$ for $C \in \mathcal{C}$ form a partition of $X$, and $|Y|=o(n)$;
\item[(c)] for every $C \in \mathcal{C}$, $|Y \cap V(C)| = o(|V(C)|)$ and the induced subgraph $G[X\cap V(C)]$ contains at least $(1-o(1))|V(C)|$ vertices of degree at least $k/8$;
\item[(d)] for every $C \in \mathcal{C}$ and every vertex $v \in X \cap V(C)$, there are at most $o(k)$ edges of $G$ incident to $v$ whose other endpoint lies in $X \setminus V(C)$.
\end{itemize}
\end{lem}
\begin{proof}
Let $t=(\log\omega)^{1/5}$ and $V=V(G)$.
A straightforward
application of Chernoff's inequality and of the union bound shows that $G_{p}$
a.a.s. satisfies the following property: ``for every pair of sets $A$
and $B$ that have $e_{G}(A,B)\ge\frac{nk}{\omega^{1/2}}$, we have
$e_{G_{p}}(A,B)\ge\frac{1}{2}n\omega^{1/2}$''. Expose $G_{p}$ and
condition on this event.

Let $C$ be a graph defined over a subset of vertices of $V$.
For an edge $e$ of $G$, we say that $e$ is \emph{covered} by
$C$ if both of the endpoints of $e$ belong to
$V(C)$
(note that this does not necessarily imply that $e$ is an edge of $C$).
Let $\mathcal{C}_{0}$ be a collection of $t$-connected edge-disjoint subgraphs
of $G_{p}$ of order at least $t^{4}$ each, which maximizes the total
number of edges covered and whose sum of orders is minimized.
Note that for every pair of
graphs $C,C'\in\mathcal{C}_{0}$, we have $|V(C)\cap V(C')|<t$
as otherwise by Lemma \ref{lem:highconnectedgraph} (iii) we can replace
the two graphs $C$ and $C'$ in $\mathcal{C}_{0}$ by a single graph
$C\cup C'$ in order to find a collection of $t$-connected edge-disjoint
subgraphs that contradicts the minimality of sum of orders of the
collection $\mathcal{C}_{0}$. We will repeatedly apply this idea
throughout this proof; the graphs in $\mathcal{C}_{0}$ cannot be
combined to give another $t$-connected subgraph.

\medskip
\noindent \textbf{Step 1 : Initial partition}
\medskip

Let $X_{0}$ be the set of vertices which are contained in at most
$t^{4}$ graphs in $\mathcal{C}_{0}$, and let $Y_{0}=V\setminus X_{0}$.

\begin{claim}
\label{claim:1}$|Y_{0}|\le\frac{4n}{t^{3}}$.\end{claim}
\begin{proof}
Consider an auxiliary bipartite graph $\Gamma_{0}$ whose vertex
set consists of two parts, where one part is $V$ and the
other part is $\mathcal{C}_{0}$. A pair $(v,C)$ forms an edge in
$\Gamma_{0}$ if $v\in V(C)$. Suppose that $\Gamma_{0}$ contains
a $t$-connected subgraph $\Gamma_{0}'$, and let $C_{1},C_{2},\cdots,C_{s}$
be the graphs in $V(\Gamma_{0}')\cap\mathcal{C}_{0}$. We claim that
the union $C_{\sigma}=C_{1}\cup C_{2}\cup\cdots\cup C_{s}$
forms a $t$-connected graph; thus deducing a contradiction to the minimality
assumption of $\mathcal{C}_{0}$. Indeed, suppose that we removed a set $S$ of $t-1$
vertices from $C_{\sigma}$, and let $v,w$ be two vertices which
have not been removed. Since $\Gamma_{0}'$ is $t$-connected, after
removing $S$ from the $V$ part of the graph $\Gamma_{0}'$, we still
have a connected graph, and thus without loss of generality we can find a
path of the form $(C_{1},y_{1},C_{2},y_{2},\cdots,y_{h-1},C_{h})$
in the graph $\Gamma_0'$ where $v\in C_{1}$ and $w\in C_{h}$. Let $y_{0}=v$
and $y_{h}=w$. For each $i$, we have $y_{i-1},y_{i}\in V(C_{i})$,
and since $C_{i}$ is $t$-connected, there exists a path from $y_{i-i}$
to $y_{i}$ in the graph $C_{i}\setminus S$. By combining these
paths, we can find a path from $y_{0}=v$ to $y_{h}=w$ in $C_{\sigma}\setminus S$.

As mentioned above, this implies that we cannot have a $t$-connected
subgraph of $\Gamma_{0}$. Since each vertex in $Y_0$ is contained in
more than $t^4$ graphs in $\mathcal{C}_0$, and each graph in $\mathcal{C}_0$
is of order at least $t^4$, the number of edges of $\Gamma_{0}$
is at least $\frac{1}{2}\Big(t^{4}|Y_{0}|+t^{4}|\mathcal{C}_{0}|\Big)$,
by Lemma \ref{lem:highconnectedgraph} (iv) we have
\[
\frac{1}{2}\Big(t^{4}|Y_{0}|+t^{4}|\mathcal{C}_{0}|\Big)\le2t(n+|\mathcal{C}_{0}|),
\]
from which it follows that $|Y_{0}|\le\frac{4n}{t^{3}}$.
\end{proof}

Let $X_{0}'$ be the
subset of vertices $v\in X_{0}$ for which there are at least $\frac{k}{\omega^{1/4}}$
edges in $G[X_{0}]$ incident to $v$ that are not covered by any of
the graphs $C\in\mathcal{C}_{0}$.

\begin{claim}
\label{claim:2}$|X_{0}'|<\frac{n}{t^{3}}$.
\end{claim}
\begin{proof}
Suppose that $|X_{0}'|\ge\frac{n}{t^{3}}$. In this case, we claim
that we can find a $t^{4}$-connected subgraph of $G_{p}$ which is
edge-disjoint from all the graphs in $\mathcal{C}_{0}$. Since a $t^{4}$-connected
graph is necessarily a $t$-connected graph with at least $t^{4}$
vertices, this will contradict the maximality of the family $\mathcal{C}_{0}$.

Color each graph in $\mathcal{C}_{0}$ by either red or blue, uniformly
and independently at random. Let $A$ be the collection of vertices
$v\in X_{0}'$ for which all the graphs in $\mathcal{C}_{0}$ that
contain $v$ are of color red, and similarly define $B$ for blue
graphs. Let $\{v,w\}$ be an edge in $G[X_{0}]$ which is not
covered by any graph in $\mathcal{C}_{0}$. Then since there are no
graphs in $\mathcal{C}_0$ containing both $v$ and $w$, the probability that $\{v,w\}$
contributes towards $e_{G}(A,B)$ is exactly $2^{-d_{v}-d_{w}}$,
where $d_{v}$ and $d_{w}$ are the numbers of graphs in $\mathcal{C}_{0}$
that contain $v$ and $w$, respectively. By the definition of $X_{0}'$,
there are at least $\frac{1}{2}|X_{0}'|\cdot\frac{k}{\omega^{1/4}}$ edges which
are not covered by any graph in $\mathcal{C}_{0}$. Since
the vertices in $X_{0}$ are covered at most $t^{4}$ times, we have
\[
\BBE[e_{G}(A,B)]\ge \frac{1}{2}|X_{0}'|\cdot\frac{k}{\omega^{1/4}} \cdot 2^{-2t^{4}}
  \ge \frac{n}{t^3} \cdot \frac{k}{\omega^{1/4} 2^{2t^{4}+1}} \ge\frac{kn}{\omega^{1/2}},
\]
where we used the fact that $t=(\log\omega)^{1/5}$. Therefore, there
exists a choice of coloring of graphs in $\mathcal{C}_{0}$ such that
$e_{G}(A,B)\ge\frac{kn}{\omega^{1/2}}$, and this implies that $e_{G_{p}}(A,B)\ge\frac{1}{2}\omega^{1/2}n$
(recall that we conditioned on this fact).
By Lemma \ref{lem:highconnectedgraph} (iv), there exists a $t^{4}$-connected
subgraph of the bipartite subgraph of $G_{p}$ induced by $A\cup B$.
Furthermore, none of the edges of this $t^{4}$-connected subgraph
could have been covered by a graph in $\mathcal{C}_{0}$. Indeed, such a
graph should be colored by both red and blue, which is impossible. Therefore, we found
a $t^{4}$-connected subgraph of $G_{p}$ as claimed.
\end{proof}
Let $\mathcal{C}_{1}=\{C\in\mathcal{C}_{0}:|C\cap X_{0}|\ge\frac{k}{t^{5}}\}$.
Our next claim establishes a useful property regarding vertices
not in $X_0'$.

\begin{claim}
\label{claim:3dot5}
For every vertex $v \in X_0 \setminus X_0'$, there are at most $o(k)$ edges of $G[X_0]$ incident to $v$
not covered by any graph in $\mathcal{C}_1$.
\end{claim}
\begin{proof}
Note that there are two possible circumstances in which an edge in $G[X_0]$
is not covered by some graph in $\mathcal{C}_1$. First is if it is not
covered by any graph in $\mathcal{C}_0$, and second is if it is covered
by some graph in $\mathcal{C}_0 \setminus \mathcal{C}_1$. For a
fixed vertex $x \in X_0 \setminus X_0'$, by the definition of the set $X_0'$,
there are at most $\frac{k}{\omega^{1/4}} =o(k)$ edges incident to $x$ of the first type.
Also, since $x$ is contained in at most $t^4$ graphs in $\mathcal{C}_0$
and each graph $C \in \mathcal{C}_0 \setminus \mathcal{C}_1$ satisfies $|X_0 \cap V(C)| \le \frac{k}{t^5}$, there are at most $t^4 \cdot \frac{k}{t^5} = o(k)$
edges incident to $x$ of the second type. Thus we establish our claim.
\end{proof}

Let $X_{0}''$ be the set of vertices $v\in X_{0}$ which are
covered by at least two graphs in $\mathcal{C}_{1}$, or
are not covered by any graph in $\mathcal{C}_1$.
We defer the proof of the following claim, which is somewhat similar to that
of Claim \ref{claim:1}, to later.
\begin{claim}
\label{claim:3}$|X_{0}'' \setminus X_0'|<\frac{21n}{t^{3}}$.
\end{claim}

\medskip
\noindent \textbf{Step 2 : Intermediate partition}
\medskip

Let $X_{1}=X_{0}\setminus(X_{0}'\cup X_{0}'')$
and $Y_{1}=V\setminus X_{1}=Y_0\cup(X_{0}'\cup X_{0}'')$.
We first verify that
the partition $V = X_1 \cup Y_1$ and the collection
of graphs $\mathcal{C}_1$ satisfy the following list of properties
from the statement of Lemma \ref{lem:structure1}:

\begin{itemize}
  \setlength{\itemsep}{1pt} \setlength{\parskip}{0pt}
  \setlength{\parsep}{0pt}
\item[(a)] every graph $C \in \mathcal{C}_1$ is $(\log \omega)^{1/5}$-connected;
\item[(b)] the sets $X_1 \cap V(C)$ for $C \in \mathcal{C}_1$ form a partition of $X_1$, and
$|Y_1| < \frac{26n}{t^3}=o(n)$;
\item[(d)] for every $C \in \mathcal{C}_1$ and every vertex $v \in X_1 \cap V(C)$, there are at most $o(k)$ edges of $G$ incident to $v$ whose other endpoint lies in $X_1 \setminus V(C)$.
\end{itemize}
Property (a) follows from the definition of $\mathcal{C}_0$. Property (b) follows
from the definition of $X_0''$ and Claims
\ref{claim:1}, \ref{claim:2} and \ref{claim:3} which imply that $|Y_{1}|<\frac{26n}{t^{3}}$.
Property (d) follows from Claim \ref{claim:3dot5}, Property (b), and the fact that
$X_1 \subset X_0 \setminus X_0'$.

In order to find a partition of the vertex set and a collection of graphs satisfying
Property (c) as well, we will identify the graphs $C \in \mathcal{C}_1$ that do not
satisfy Property (c), and will move the vertices of $X_1 \cap V(C)$ to $Y_1$. Note
that this adjustment does not affect Properties (a) and (d). Our goal is to
maintain Property (b) as well by keeping the total number of vertices that
we move small enough.

Let $X_{1}'$ be the subset of vertices
of $X_{1}$ which have at least $\frac{3k}{4}$ neighbors in the set
$Y_{1}$. If $|X_{1}'|\ge\frac{130n}{t^{3}}$, then the bipartite subgraph
induced by $X_{1}'\cup Y_{1}$ has at most $\frac{6}{5}|X_{1}'|$
vertices and at least $\frac{3k}{4}|X_{1}'|$ edges. Thus the
average degree of this graph is at least $2\cdot\frac{3k}{4}\cdot\frac{5}{6}=\frac{5k}{4}$,
which contradicts our assumption saying that $G$ does not contain
such a subgraph. Therefore we have 
\begin{equation}
|X_{1}'|<\frac{130n}{t^{3}}. \label{eq:x1_prime}
\end{equation}

\begin{claim} \label{claim:coverededges}
For a vertex $x \in X_1 \setminus X_1'$ contained in $C_x \in \mathcal{C}_1$,
$x$ has degree at least $\frac{k}{8}$ in the subgraph of $G$ induced by the vertex set $X_1 \cap V(C_x)$.
\end{claim}
\begin{proof}
For a vertex $x\in X_{1}\setminus X_{1}'$, let $C_x \in \mathcal{C}_1$ be the
graph containing $x$. Since $x \notin X_1'$, there are at least $\frac{k}{4}$
edges of $G$ incident to $x$ in $G[X_{1}]$. By Property (d), at most
$o(k)$ edges among them are incident to a vertex not in $C_x$.
Therefore, $x$ has degree at least $\frac{k}{4} - o(k) \ge \frac{k}{8}$ in the
subgraph of $G$ induced by $X_1 \cap V(C_x)$.
\end{proof}

Let $\mathcal{C}_{1}'=\{C\in\mathcal{C}_{1}:|V(C)\cap X_{1}'|\ge\frac{|V(C)|}{t}\}$
and $\mathcal{C}_{1}''=\{C\in\mathcal{C}_{1}:|V(C)\cap Y_{1}|\ge\frac{|V(C)|}{t}\}$.

\begin{claim}\label{claim:5}$\sum_{C\in\mathcal{C}_{1}'}|V(C)|=o(n)$.\end{claim}
\begin{proof}
By the definition of $\mathcal{C}_1'$, we have $\sum_{C\in\mathcal{C}_{1}'}\frac{|V(C)|}{t}\le\sum_{C\in\mathcal{C}_{1}'}|V(C)\cap X_{1}'|\le|X_{1}'|$. By \eqref{eq:x1_prime}, 
this implies $\sum_{C\in\mathcal{C}_{1}'}\frac{|V(C)|}{t}\le \frac{130n}{t^3}$,
from which it follows that $\sum_{C\in\mathcal{C}_{1}'}|V(C)|\le\frac{130n}{t^{2}}$.
\end{proof}

\begin{claim}\label{claim:4}$\sum_{C\in\mathcal{C}_{1}''}|V(C)|=o(n)$.\end{claim}

The proof of Claim \ref{claim:4} will be given later.

\medskip
\noindent \textbf{Step 3 : Final partition and the collection of $t$-connected subgraphs}
\medskip

Let $\mathcal{C}_{2}=\mathcal{C}_{1}\setminus(\mathcal{C}_{1}'\cup\mathcal{C}_{1}'')$.
Let $X_{2}$ be the subset of vertices of $X_{1}$ which are covered
by some graph in $\mathcal{C}_{2}$, and let $Y_{2}=V\setminus X_{2}$.
We claim that the partition $V=X_{2}\cup Y_{2}$ and the collection
$\mathcal{C}_{2}$ satisfy the claims of the lemma. We recall the
properties that we wish to establish.

\begin{itemize}
  \setlength{\itemsep}{1pt} \setlength{\parskip}{0pt}
  \setlength{\parsep}{0pt}
\item[(a)] every graph $C \in \mathcal{C}_2$ is $(\log \omega)^{1/5}$-connected;
\item[(b)] the sets $X_2 \cap V(C)$ for $C \in \mathcal{C}_2$ form a partition of $X_2$, and $|Y_2|=o(n)$;
\item[(c)] for every $C \in \mathcal{C}_2$, $|Y_2 \cap V(C)| = o(|V(C)|)$ and the induced subgraph $G[X_2\cap V(C)]$ contains at least $(1-o(1))|V(C)|$ vertices of degree at least $k/8$;
\item[(d)] for every vertex $v \in X_2 \cap V(C)$, there are at most $o(k)$ edges of $G$ incident to $v$ whose other endpoint lies in $X_2 \setminus V(C)$.
\end{itemize}

As mentioned above, Properties (a) and (d) follow from the same properties
for $X_1, Y_1$, and $\mathcal{C}_1$.
Since $|X_{2}|\ge|X_{1}|-\sum_{C\in\mathcal{C}_{1}'\cup\mathcal{C}_{1}''}|V(C)|=(1-o(1))n$,
and $|Y_{2}|=o(n)$, Property (b) follows as well.
Note that $X_1 \supseteq X_2$, and that the vertices in $X_1 \setminus X_2$
are covered exactly once by some graph in $\mathcal{C}_1$.
Therefore, for all $C \in \mathcal{C}_2$, we have
$V(C) \cap X_1 = V(C) \cap X_2$ and $V(C) \cap Y_1 = V(C) \cap Y_2$.
Thus for $C\in\mathcal{C}_{2}$, since $C \notin \mathcal{C}_1''$,
we have $|V(C)\cap Y_{2}|=|V(C)\cap Y_{1}|<\frac{|V(C)|}{t}$,
and the first part of Property (c) holds. Also,
by Claim \ref{claim:coverededges}, for $C\in\mathcal{C}_{2}$
the vertices in $V(C)\cap(X_{2}\setminus X_{1}')=V(C)\cap(X_{1}\setminus X_{1}')$
have degree at least $\frac{k}{8}$ in the subgraph of $G$
induced by $V(C) \cap X_1 = V(C) \cap X_2$.
By the fact $C \notin \mathcal{C}_1'$, we have
\[
|V(C)\cap(X_{1}\setminus X_{1}')|=(1-o(1))|V(C)\cap X_{1}|=(1-o(1))|V(C)|,
\]
and this establishes the second part of Property (c).
\end{proof}

It remains to prove Claims \ref{claim:3} and \ref{claim:4}.

\begin{proof}[Proof of Claim \ref{claim:3}]
Recall that $\mathcal{C}_{1}=\{C\in\mathcal{C}_{0}:|C\cap X_{0}|\ge\frac{k}{t^{5}}\}$
and $X_{0}''$ is the set of vertices $v\in X_{0}$ which are covered
by at least two graphs in $\mathcal{C}_{1}$, or are not covered
by any graph in $\mathcal{C}_1$. Let $X_{0, \ge 2}''$ be the
vertices which are covered by at least two graphs in $\mathcal{C}_1$
and $X_0'''$ be the vertices not covered by any graph
in $\mathcal{C}_1$.

We first estimate the size of the set $X_{0, \ge 2}''$.
Since the graphs in
$\mathcal{C}_{1}$ intersect $X_{0}$ in at least $\frac{k}{t^{5}}$
vertices and each vertex in $X_{0}$ is covered at most $t^{4}$ times,
we have $\frac{k}{t^{5}}|\mathcal{C}_{1}|\le t^{4}|X_{0}|$, from
which it follows that 
\begin{equation}
|\mathcal{C}_{1}|\le\frac{t^{9}n}{k}.  \label{eq:claim3}
\end{equation}
Consider the following auxiliary graph $\Gamma_{1}$ over the vertex set $\mathcal{C}_{1}$,
where two vertices $C, C' \in \mathcal{C}_1$ are connected by an edge
if they share a vertex from $X_{0,\ge  2}''$
(we place only one edge for each vertex even
it is contained in more than two graphs in $\mathcal{C}_1$). The number of vertices of $\Gamma_{1}$
is at most $\frac{t^{9}n}{k}$. Since every two graphs in $\mathcal{C}_1$
intersect in less than $t$ vertices, each edge of $\Gamma_1$ can
account for less than $t$ vertices of $X_{0, \ge 2}''$, and thus
the number of edges of $\Gamma_1$ is at least $\frac{|X_{0, \ge 2}''|}{t}$.

Suppose that $\Gamma_{1}$ contains a $t$-connected subgraph over
the vertices $C_{1},C_{2},\cdots,C_{s}$ of $\Gamma_{1}$. We claim
that $C_{\sigma}=C_{1}\cup\cdots\cup C_{s}$ is a $t$-connected subgraph
and this will contradict the minimality of the family $\mathcal{C}_{0}$.
It suffices to prove that $C_{\sigma}$ is connected even after removing
a set $S$ of at most $t-1$ vertices. Let $v$ and $w$ be two vertices
in $V(C_{\sigma})\setminus S$. Each vertex in $S$ corresponds to
at most one edge in the auxiliary graph $\Gamma_{1}$, and thus even
after removing the edges corresponding to vertices in $S$, without
loss of generality there exists a path $(C_{1},C_{2},\cdots,C_{h})$
of $\Gamma_{1}$ for which $v\in C_{1}$ and $w\in C_{h}$. By the
definition of the graph $\Gamma_{1}$, for each $i$, there exists
a vertex $v_{i}\in C_{i}\cap C_{i+1}$ such that $v_{i}\notin S$.
Let $v_{0}=v$ and $v_{h}=w$. Then for all $0\le i<h$, we can find
a path from $v_{i}$ to $v_{i+1}$ in the graph $C_{i}\setminus S$
(recall that $C_{i}$ is $t$-connected). This implies that there
exists a path from $v$ to $w$ in $C_{\sigma}\setminus S$.

Thus $\Gamma_{1}$ cannot contain a $t$-connected subgraph. By Lemma
\ref{lem:highconnectedgraph} (iv), we then have
\[
\frac{|X_{0, \ge 2}''|}{t}\le2t\cdot\frac{t^{9}n}{k},
\]
which implies $|X_{0, \ge 2}''|\le\frac{2t^{11}n}{k}<\frac{n}{t^3}$ (note
that $t=(\log\omega)^{1/5}\le(\log k)^{1/5})$.

Now consider the set $X_0'''$. By Claim \ref{claim:3dot5} and the definition
of $X_0'''$, each vertex in $Z = X_0''' \setminus X_0'$ has at least
$k-o(k)$ neighbors in the set $Y_0$. Therefore, if $|Z| \ge 5|Y_0|$,
then we obtain a bipartite subgraph of $G$ with at least $|Z| \cdot (k-o(k))$
edges and at most $\frac{6}{5}|Z|$ vertices. Thus this bipartite graph
has average degree at least $2 \cdot (k-o(k)) \frac{5}{6} \ge \frac{5}{4}k$.
However, this contradicts our assumption, and thus we have $|Z| < 5|Y_0| \le \frac{20n}{t^3}$.
Therefore, $|X_0'' \setminus X_0'| \le |X_{0, \ge 2}''| + |X_0'''\setminus X_0'| \le \frac{21n}{t^3}$.
\end{proof}

\begin{proof}[Proof of Claim \ref{claim:4}]
Recall that $\mathcal{C}_{1}''=\{C\in\mathcal{C}_{1}:|V(C)\cap Y_{1}|\ge\frac{|V(C)|}{t}\}$.
Consider an auxiliary bipartite graph $\Gamma_{2}$ whose vertex set
consists of two parts, where one part is the set $Y_{1}$ and the
other part is $\mathcal{C}_{1}''$. A pair $(v,C)$ forms an edge in
$\Gamma_{2}$ if $v\in V(C)$. As we have seen in the proof of Claim
\ref{claim:1}, this graph cannot contain a $t$-connected subgraph
(in fact, $\Gamma_{2}$ is a subgraph of $\Gamma_{0}$ defined in
the proof of Claim \ref{claim:1}). By Property (b) in Step 2 
which follows from Claims \ref{claim:1}, \ref{claim:2} and \ref{claim:3}),
we have $|Y_{1}|\le\frac{26n}{t^{3}}$, and by \eqref{eq:claim3} we have
$|\mathcal{C}_{1}''|\le |\mathcal{C}_{1}|\le\frac{t^{9}n}{k}$.
Hence the number of vertices of $\Gamma_{2}$ is $|Y_{1}|+|\mathcal{C}_{1}''|\le\Big(\frac{26}{t^{3}}+\frac{t^{9}}{k}\Big)n$.
The number of edges is at least $\sum_{C\in\mathcal{C}_{1}''}|V(C)\cap Y_{1}|\ge\sum_{C\in\mathcal{C}_{1}''}\frac{|V(C)|}{t}$.
Therefore by Lemma \ref{lem:highconnectedgraph} (iv), we have
\[
\sum_{C \in \mathcal{C}_1''} \frac{|V(C)|}{t}\le2t\cdot\Big(\frac{26}{t^{3}}+\frac{t^{9}}{k}\Big)n,
\]
which implies that $\sum_{C\in\mathcal{C}_{1}''}|V(C)|<\frac{53n}{t}$ (recall that $t = (\log w)^{1/5}
\le (\log k)^{1/5}$).
\end{proof}

One more round of sprinkling will give us our desired structural lemma,
Lemma \ref{lem:structure}, which says the following.
Let $0<\varepsilon\le\frac{1}{2}$ be fixed, $G$ be a graph on $n$ vertices
with minimum degree at least $k$ that does not contain
a bipartite subgraph of average degree at least $\frac{5}{4}k$ and let $p=\frac{\omega}{k}$
for some function $\omega=\omega(k) \ll k$
that tends to infinity as $k$ does. Then
$G_{p}$ a.a.s. admits a partition $V=X\cup Y$ of its vertex set,
and contains a collection $\mathcal{C}$ of subgraphs of $G_{p}$
satisfying:
\begin{itemize}
  \setlength{\itemsep}{1pt} \setlength{\parskip}{0pt}
  \setlength{\parsep}{0pt}
\item[(a)] for every $C \in \mathcal{C}$, $C$ is $(\log \omega)^{1/5}$-connected;
\item[(b)] the sets $X \cap V(C)$ for $C \in \mathcal{C}$ form a partition of $X$, and $|Y|=o(n)$;
\item[(c)] for every $C \in \mathcal{C}$, one of the following holds:
\begin{itemize}
\item[(i)] the graph $G[V(C)]$ contains at least $(1-\varepsilon)|V(C)|$ vertices of degree at least $(1-\varepsilon)k$, or
\item[(ii)] $|Y \cap V(C)| = o(|V(C)|)$, the graph $G[X \cap V(C)]$ contains at least $(1-\varepsilon)|V(C)|$ vertices
of degree at least $\frac{k}{8}$, and there exists a bipartite graph $\Gamma_C \subseteq G$ with parts $X \cap V(C)$ and
$Y \setminus V(C)$ which contains at least $\frac{|V(C)|\varepsilon^2 k}{4}$ edges and
has maximum degree at most $\frac{8}{\varepsilon}k$.
\end{itemize}
\end{itemize}
\begin{proof}[Proof of Lemma \ref{lem:structure}]
Set $p_{1}=p_{2}=\frac{\omega}{2k}$ and $t=(\log\omega)^{1/5}$.
Apply Lemma \ref{lem:structure1} to $G_{p_{1}}$ to find a partition
$V=X\cup Y$ and a collection $\mathcal{C}$ of subgraphs of $G_{p_1}$
satisfying the following properties:
\begin{itemize}
  \setlength{\itemsep}{1pt} \setlength{\parskip}{0pt}
  \setlength{\parsep}{0pt}
\item[(a')] every graph $C \in \mathcal{C}$ is $(\log \omega)^{1/5}$-connected;
\item[(b')] the sets $X \cap V(C)$ for $C \in \mathcal{C}$ form a partition of $X$, and $|Y|=o(n)$;
\item[(c')] for every $C \in \mathcal{C}$, $|Y \cap V(C)| = o(|V(C)|)$ and the induced subgraph $G[X\cap V(C)]$ contains at least $(1-o(1))|V(C)|$ vertices of degree at least $k/8$;
\item[(d')] for every $C \in \mathcal{C}$ and every vertex $v \in X \cap V(C)$, there are at most $o(k)$ edges of $G$ incident to $v$ whose other endpoint lies in $X \setminus V(C)$.
\end{itemize}
(we denote the properties by (a'), (b'), (c'), and (d') in order to distinguish
it from the properties (a), (b), and (c)).

For $C \in \mathcal{C}$, let $X_C = X \cap V(C)$, $Z_C = Y \setminus V(C)$.
Define a bipartite subgraph $\Gamma_C$ of $G$ with bipartition
$X_C \cup Z_C$ as follows:
first take all the edges of $G$ between $X_C$ and $Z_C$, and for
each vertex of $X_C$ of degree at least $k$, retain $k$ arbitrarily
chosen edges incident with it.
Let $Z_{C}' \subset Z_C$ be the vertices which have degree
greater than $\frac{8k}{\varepsilon}$ in this bipartite subgraph, and let $Z_{C}'' \subset Z_C$ be
the vertices which have degree at most $\frac{8k}{\varepsilon}$.

Now expose the edges of $G_{p_2}$.
If a vertex $z\in Z_{C}'$ has at least $t$ neighbors in $G_{p_{2}}$
in the set $X_{C}$, then we can add $z$ to the graph $C$ to obtain
another $t$-connected subgraph (see Lemma \ref{lem:highconnectedgraph}
(ii)). In such a situation, we say that $z$ is \emph{absorbed} to $C$,
and let the \emph{enlarged} graph $\hat{C}$ be the union of $C$ with the
set of vertices which is absorbed by $C$.
Note that even though the same holds for vertices in $Z_C''$, for technical
reasons, we only absorb vertices from $Z_C'$ to $C$.
Further note that we allow a fixed vertex being absorbed to several graphs,
and that this does not affect the property that $X \cap V(C)$ forms a
partition of $X$, since each vertex being absorbed is a vertex in $Y$.

Let $\mathcal{C}'$ be the collection of graphs $C \in \mathcal{C}$ for which
the number of edges of $\Gamma_C$ incident to $Z_C'$ which are not covered
by the enlarged graph $\hat{C}$ is at least $\frac{\varepsilon^{2}k}{8}|V_{C}|$.

\begin{claim} \label{claim:last}
We a.a.s. have $\sum_{C\in\mathcal{C}'} |X \cap V(C)|=o(n)$.
\end{claim}
\begin{proof}
Suppose that the vertices in $Z_{C}'$ have degree $d_{1},\cdots,d_{s}$
in $\Gamma_C$. Since we only consider at most $k$ edges incident
to each vertex of $X_{C}$, we have $\sum_{i}d_{i}\le k|V(C)|$ (also
note that $d_i \le |V(C)|$ for all $i$).
For a vertex $z\in Z_{C}'$, since $z$ has degree $d_{z}\ge \frac{8k}{\varepsilon}$
in $\Gamma_C$, by Chernoff's inequality,
the probability that $z$ cannot be absorbed is
at most $e^{-\Omega(\omega)}$. Let $N$ be the random variable
which counts the number of edges of $\Gamma_C$ incident to non-absorbed vertices
from $Z_C'$ after exposing $G_{p_2}$. We have,
\[
\BBE[N] = \sum_{i}d_{i}\cdot e^{-\Omega(\omega)}\le\Big(\sum_{i}d_{i}\Big)\cdot e^{-\Omega(\omega)}=o(|V(C)|\cdot k).
\]
Let ${\bf 1}_{i}$ be the indicator random variable
of the event that the $i$-th vertex of $Z_{C}'$ is
absorbed to $C$. Note that the events ${\bf 1}_i$ are independent since they
depend on disjoint sets of edges, and that we have $N = \sum_{i=1}^{s} d_i \cdot {\bf 1}_{i}$.
Since $0 \le \frac{d_i}{|V(C)|} \le 1$, by applying Hoeffding's
inequality to the random variables $\frac{d_{i}}{|V(C)|}\cdot {\bf 1}_{i}$,
we see that the probability of $\frac{N}{|V(C)|} \ge \frac{\varepsilon^2k}{8}$,
which is equivalent to $C \in \mathcal{C}'$, is
at most $e^{-\Omega(k)}$. Then,
\[
\BBE\Big[\sum_{C\in\mathcal{C}'}|X\cap V(C)|\Big]\le\sum_{C\in\mathcal{C}}|X\cap V(C)|\cdot e^{-\Omega(k)}=o(n).
\]
Thus by Markov's inequality, it follows that
$\sum_{C \in \mathcal{C}'} |X \cap V(C)| = o(n)$ a.a.s.
\end{proof}

Condition on the conclusion of Claim \ref{claim:last}.
Let $\mathcal{C}_1 = \{ \hat{C} : C \in \mathcal{C} \setminus \mathcal{C}' \}$
(recall that $\hat{C}$ is the enlarged graph obtained from $C$).
Let $X_1$ be the subset of
vertices of $X$ covered by some graph in $\mathcal{C}_1$, and let
$Y_1 = V \setminus X_1$. We claim that the partition
$V = X_1 \cup Y_1$ and the collection of graphs $\mathcal{C}_1$ satisfy
properties (a), (b), (c) of Lemma \ref{lem:structure} (which we listed before this proof).

Property (a) immediately follows from how we constructed the enlarged graphs.
Note that the difference between the sets $X$ and $X_1$ consist of the
vertices of $X \cap V(C)$ for $C \in \mathcal{C}'$, and
that $\sum_{C \in \mathcal{C}'} |X \cap V(C)| = o(n)$. Since the
difference between a graph $C \in \mathcal{C}$ and its
enlarged graph $\hat{C}$ lie in $Y \subset Y_1$, Property (b)
follows from Property (b'). We now focus on proving that (c) holds as well.

Take a graph $C \in \mathcal{C} \setminus \mathcal{C}'$.
If $e_{\Gamma_C}(X_C, Z_C'') \ge \frac{\varepsilon^2 k}{4}|V(C)|$, then
(ii) holds and there is nothing to prove (recall that the vertices
in $Z_C''$ are not added to the enlarged graph).
Suppose that $e_{\Gamma_C}(X_C, Z_C'') < \frac{\varepsilon^2 k}{4}|V(C)|$.
Since $C\notin\mathcal{C}'$, there are less than
$\frac{\varepsilon^{2}k}{8}|V(C)|$ edges of $\Gamma_C$ incident to
$Z_C'$ that are not covered by $\hat{C}$. Therefore, the total number of edges in $\Gamma_C$ not covered by $\hat{C}$ is at most
$e_{\Gamma_C}(X_C, Z_C'')+\frac{\varepsilon^{2}k}{8}|V(C)| \le \frac{3\varepsilon^{2}k}{8}|V(C)|$.

We can count the number of such edges in another way.
Let $X_C'$ be the subset of vertices of $X_C$, whose degree in $G[V(\hat{C})]$
is less than $(1-\varepsilon)k$. Since $X_1 \subset X$,
by Property (d'), a vertex in $X_C'$ can
have at most $o(k)$ neighbors in $X_1 \setminus V(C)$. Therefore, the number of
edges of $\Gamma_C$ not covered by $\hat{C}$ is at least
$|X_C'| \cdot \frac{\varepsilon}{2}k$. By combining this with the
bound established above, we have
\[ |X_C'| \cdot \frac{\varepsilon}{2}k \le \frac{3\varepsilon^{2}k}{8}|V(C)|, \]
from which it follows that $|X_C'| \le\frac{3\varepsilon}{4} |V(C)|$.
Recall that by Property (c'), we have $|X \cap V(C)| = (1-o(1))|V(C)|$ for all
$C \in \mathcal{C}$.
Thus $G[V(\hat{C})]$ contains at least $|X \cap V(C)| - \frac{3\varepsilon}{4} |V(C)| \ge \left(1 - \frac{7\varepsilon}{8}\right)|V(C)|$ vertices
of degree at least $(1-\varepsilon)k$. It then suffices to prove that
$|V(C)| \ge \left(1- \frac{\varepsilon}{8}\right)|V(\hat{C})|$. Since
we only added the vertices of $Z'_C$ to $C$ in order to obtain $\hat{C}$,
we have
\[
|V(\hat{C}) \setminus V(C)|
 \le |Z_C'| \le \frac{e(\Gamma_C)}{(8/\varepsilon)k}
 \le \frac{k|V(C)|}{(8/\varepsilon) k} = \frac{\varepsilon}{8}|V(C)|, \]
and it implies $|V(\hat{C})| \le \left(1 + \frac{\varepsilon}{8}\right)|V(C)|
\le \frac{1}{1-(\varepsilon/8)}|V(C)|$.
\end{proof}

\section{Concluding remarks}

In this paper, we studied random subgraphs of graphs with large
minimum degree. Our goal was to extend classical results on random
graphs to a more general setting, where we replace the host graph by
a graph with large minimum degree. We proved that the results
asserting the a.a.s. existence of long paths and cycles in $G(n,p)$
can in fact be extended to this setting.
The problems we addressed in this paper are also closely related to
our previous paper \cite{KrLeSu}, where we studied random subgraphs
of graphs on $n$ vertices with minimum degree at least
$\frac{n}{2}$, and proved that for every graph $G$ of minimum degree
at least $\frac{n}{2}$ and $p\gg\frac{\log n}{n}$, the random graph
$G_{p}$ a.a.s. is Hamiltonian.

Similarly to Theorem \ref{thm:hamiltonpath}, it is natural to expect
that for every graph $G$ of minimum degree at least $k$ and
$p\ge\frac{(1+\varepsilon)\log k}{k}$, the graph $G_{p}$ a.a.s.
contains a cycle of length at least $k+1$. While we are unable to
settle this question at present, it seems that the techniques we
developed in this paper can be useful in attacking this problem.

It is also known that a directed graph of minimum outdegree at least
$k$ contains a cycle of length at least $k+1$. However, it is no
longer true that there exists a function $p_0 = p_0(k) < 1$ for
which the following holds: if $p \ge p_0$, then for every directed
graph $D$ of minimum outdegree at least $k$, $D_p$ a.a.s. contains a
cycle of length $k$. Indeed, suppose that we are given a function
$p_0$ depending only on $k$. Let $N$ be a large enough integer
depending on $p_0$, and consider a blow-up of a directed cycle of
length $N$, where each vertex is replaced by an independent set of
size $k$, and each edge is replaced by a complete bipartite graph,
whose orientation of edges comes from that of the underlying edge in
the directed cycle (call this directed graph $D$). A necessary
condition for $D_{p_0}$ to contain a cycle is that each complete
bipartite graph contains at least one edge. The probability of this
happening is exactly $(1 - (1-p_0)^{k^2})^{N}$. However, this can be
made arbitrarily small by choosing $N$ to be large enough depending
on $p_0$. Note that if the above event does not hold, then not only
does $D_{p_0}$ not contain a cycle of length $k$, but it also does
not contain a cycle of any length. This gives a partial explanation
to why the proof of Theorem \ref{thm:linearcycle} is unexpectedly
challenging technically, as many ``natural" approaches at one point
reduce the problem to a problem of finding a cycle in some directed
graph after taking a random subgraph of it.

\medskip

\noindent \textbf{Acknowledgement}. We would like to thank the
two anonymous referees for their valuable comments.


\begin{thebibliography}{100}
\bibitem{AjKoSz}M. Ajtai, J. Koml\'os, and E. Szemer\'edi, The
longest path in a random graph, \emph{Combinatorica} 1 (1981), 1--12.

\bibitem{BeKrSu}I. Ben-Eliezer, M. Krivelevich, and B. Sudakov, Long
cycles in subgraphs of (pseudo)random directed graphs, \emph{J.
Graph Theory} 70 (2012), 284--296.

\bibitem{Bollobas}B. Bollob\'as, The evolution of sparse graphs,
\emph{Graph Theory and Combinatorics (Proc. Cambridge
Combinatorics Conference in Honors of Paul Erd\H{o}s}, Academic Press (1984), 35--57.

\bibitem{Bollobas2}B. Bollob\'as, \textbf{Random graphs}, Cambridge
Stud. Adv. Math. 73, Cambridge University Press, Cambridge (2001).

\bibitem{BoFeFr}B. Bollob\'as, T. Fenner, and A. Frieze,
Long cycles in sparse random graphs, in \emph{Graph theory and combinatorics} (Cambridge, 1983), Academic Press, London (1984), 59--64.

\bibitem{CuKa}B. Cuckler and J. Kahn, Hamiltonian cycles in Dirac graphs, \emph{Combinatorica} 29 (2009), 299--326.

\bibitem{Diestel}R. Diestel, \textbf{Graph theory}, Volume 173 of
\emph{Graduate Texts in Mathematics}, Springer-Verlag, Heidelberg, 4th
edition (2010).

\bibitem{Dirac}G. Dirac, Some theorems on abstract graphs, \emph{Proc.
London Math. Soc.}, 2 (1952), 69--81.


\bibitem{Vega} W. Fernandez de la Vega, Long paths in random graphs,
\emph{Studia Sci. Math. Hungar.} 14 (1979), 335--340.

\bibitem{FrPi}J. Friedman and N. Pippenger, Expanding graphs contain
all small trees, \emph{Combinatorica} 7 (1987), 71--76.

\bibitem{Frieze}A. Frieze, On large matchings and cycles in sparse
random graphs, \emph{Discrete Math.} 59 (1986), 243--256.

\bibitem{Gilbert}E. Gilbert, Random graphs, \emph{Annals of Mathematical Statistics} 30 (1959), 1141--1144.

\bibitem{KoSz}J. Koml\'os and E. Szemer\'edi, Limit distribution
for the existence of Hamilton cycles in random graphs, \emph{Discrete
Math.} 43 (1983), 55--63.

\bibitem{Korshunov}A. Korshunov, Solution of a problem of Erd\H{o}s
and R\'enyi on Hamilton cycles in non-oriented graphs, \emph{Soviet
Math. Dokl.}, 17 (1976), 760--764.

\bibitem{KrLeSu}M. Krivelevich, C. Lee, and B. Sudakov, Robust Hamiltonicity
of Dirac graphs, \emph{Trans. Amer. Math. Soc.}, in press.

\bibitem{KrSu}M. Krivelevich and B. Sudakov, The phase transition
in random graphs -- a simple proof, \emph{Random Struct. Algor.}, in press.

\bibitem{Locke}S. Locke, Relative lengths of paths and cycles in
$k$-connected graphs, \emph{J. Comb. Theory B} 32 (1982), 206--222.

\bibitem{Mader}W. Mader, Homomorphies\'atze f\"ur Graphen, \emph{Math.
Ann.} 178 (1968), 154--168.

\bibitem{Menger}K. Menger, Zur allgemeinen Kurventheorie, \emph{Fund.
Math.} 10 (1927), 96--115.

\bibitem{McDiarmid}C. McDiarmid, Concentration, in \emph{Probabilistic
methods for algorithmic discrete mathematics} (ed. M. Habib, C. McDiarmid,
J. Ramirez-Alfonsin, and B. Reed), Springer, Berlin (1998), 1--46.

\bibitem{Posa}L. P\'osa, Hamiltonian circuits in random graphs,
\emph{Discrete Math.} 14 (1976), 359--364.

\bibitem{SaSeSz}G. S\'ark\"ozy, S. Selkow, and E. Szemer\'edi, On the number of Hamiltonian cycles in Dirac graphs, \emph{Discrete Math.} 265 (2003), 237--250.

\bibitem{SuVu}B. Sudakov and V. Vu, Local resilience of graphs, \emph{Random Struct. Algor.} 33 (2008), 409--433.
\end{thebibliography}
\end{document}